\RequirePackage{fix-cm}

\documentclass[smallextended]{svjour3}
\usepackage{mathptmx}  
\usepackage{amsmath,amssymb,amsfonts}    

\smartqed  

\spdefaulttheorem{assumption}{Assumption}{\bf}{}
\spdefaulttheorem{prty}{Property}{\bf}{\it}

\begin{document}

\newcommand{\limit}{\mathop{\rm \rightarrow}\limits}
\newcommand{\weaklimit}{\mathop{\rm \Rightarrow}\limits}
\newcommand{\equal}{\mathop{\rm =}\limits}
\newcommand{\lessequal}{\mathop{\rm \leq}\limits}
\newcommand{\greaterequal}{\mathop{\rm \geq}\limits}

\title{Fluid Limits for an~ALOHA-type Model \\ with Impatient Customers \thanks{The research of M. Frolkova and B. Zwart is supported by an NWO VIDI grant.}}

\author{Maria Frolkova \and Sergey Foss \and Bert Zwart}
\authorrunning{M. Frolkova, S. Foss, B. Zwart}

\institute{
M. Frolkova 
\at CWI, P.O. Box 94079, 1098 XG Amsterdam, The Netherlands \\
\email{M.Frolkova@cwi.nl}
\and 
S. Foss
\at Heriot-Watt University, Department of Actuarial Mathematics and Statistics, EH14 4AS Edinburgh, UK, and Institute of Mathematics, 
Novosibirsk, Russia \\  
\email{s.foss@hw.ac.uk}
\and
B. Zwart
\at CWI, EURANDOM, VU University Amsterdam, and Georgia Institute of Technology \\
\email{Bert.Zwart@cwi.nl}
}

\date{Received: April 10, 2010 / Accepted: February 10, 2012}

\maketitle

\begin{abstract}
Random multiple-access protocols of type ALOHA are used to regulate networks with a star configuration where client nodes talk to the hub node at the same frequency (finding a wide range of applications among telecommunication systems, including mobile telephone networks and WiFi networks). Such protocols control who talks at what time sharing the common idea ``try to send your data and, if your message collides with another transmission, try resending later''.

In the present paper, we consider a time-slotted ALOHA model where users are allowed to renege before
transmission completion. We focus on the scenario that leads to overload in the absence of impatience. Under mild assumptions,
we show that the fluid (or law-of-large-numbers) limit of the system workload coincides
a.s. with the unique solution to a~certain integral equation. We also demonstrate that the fluid limits
for distinct initial conditions converge to the same value as time
tends to infinity.

\end{abstract}

\section{Introduction}
ALOHA-type algorithms are intended to govern star networks in
which multiple client machines send data packets to the hub
machine at the same frequency. Thus, collisions of packets being
transmitted simultaneously are possible (clients know nothing about each other's intentions to transmit data and can not prevent collisions). Such algorithms
assume the following acknowledgment mechanism. If data has been received correctly at
the hub, which is possible only if no collisions occurred during
its transmission, then a~short acknowledgment packet is sent to
the client. If a~client has not received an~acknowledgment after a
short wait time, then it retransmits the packet after waiting a
randomly selected time interval with distribution specified by the ALOHA protocol that governs the network.

The pioneering ALOHA computer network, also known as the ALOHAnet,
was developed at the university of Hawaii under the leadership of
Norman Abramson (see~\cite{Abramson}, where Abramson first proposed
the ALOHA multi-access algorithm). The goal was to use low-cost
commercial radio equipment to connect users on Oahu and the other
Hawaiian islands with the central computer on the main Oahu campus.
The ALOHAnet became operational in~1971, providing the first
demonstration of a~wireless data network. Nowadays the ALOHA
random access techniques are widely used in WiFi, mobile telephone
networks and satellite communications.

To give an~example, we describe the conventional centralised
time-slotted ALOHA model. Here ``slotted time'' means that users
enter the system and abandon it, initiate and finish transmissions
at times $n=1,2,\ldots$. The arrival process forms an~i.i.d.
sequence $\{\xi(n); \ n \geq 1\}$; all service times are assumed to
equal $1$. ``Centralised model'' means that the total number of
users in the system is always known. Let $W(n)$ denote the total
number of users at time $n$. For any $n$, at the beginning of the
$n$-th time slot, which is the time interval $[n,n+1)$, each of the
$W(n)$ customers present in the system starts his transmission
with probability $p(n)$ (and does not with probability $1-p(n)$) independently of the others. If two or
more users attempt transmissions simultaneously, then the
transmissions collide, and hence fail, causing the users to remain in the
system in order to retransmit later. After a~successful
transmission the user leaves immediately. Note that, for any
time slot, given there are $m$ customers each starting his
transmission with probability $p$, the probability of a~successful
transmission equals $mp(1-p)^{m-1}$ and is maximised by $p=1/m$.
So we assume that $p(n)=1/W(n)$. In such a~setting, the population process $\{ W(n); \ n \geq 0 \}$ forms a~Markov chain that is positive recurrent if $\vec{E}\xi(1)<e^{-1}$ (the system is stable) and transient if
$\vec{E}\xi(1)>e^{-1}$ (the system is unstable). Stability conditions for other ALOHA-type models can be found in~\cite{survey,Hajek,Mikhailov}.

In the present paper, we extend the framework described above by
allowing impatience of users. Each user knows how long he can stay
in the system and abandons the system as soon as he succeeds to
transmit his packet or his patience time expires, depending on what
happens earlier. To distinguish between different levels of patience we
assume multiple customer classes. We also assume that the input process is non-trivial, i.e. $\vec{P}\{\xi(1) \geq 2 \}>0$, and that
$\vec{E}\xi(1)>e^{-1}$. The latter condition would mean the overload
regime if not for impatience of customers. 

Abandonment before service completion is a~very natural behavior in the case when waiting time exceeds some threshold because dwelling in the system
means spending resources, and whatever one can imagine them to be:
time itself, money, etc., there is always a~limited amount of them
at disposal. As for overload causing long waiting times, a~system
may experience it even if it is not expected to, due to
fluctuations of the actual traffic from the designed traffic.
These are the reasons of our interest in combination of the
overload regime and impatience of customers.

The results of the paper concern fluid limits for the workload
process, i.e. weak limits that arise under a~law-of-large-numbers scaling. We propose a~fluid analogue of the
stochastic model and, under mild assumptions, show that the
family of the scaled workload processes is relatively compact with weak
limit points a.s. satisfying the fluid model equation.
Besides the last argument, we provide an~independent analytical
proof of existence of a~fluid model solution for any initial state,
and show that it is unique. We also demonstrate that there exists
a unique equilibrium point for the fluid model equation, and that
all fluid limits stabilise to this point as time tends to
infinity.

One of possible generalisations of the model treated in the
present paper is to allow interference of transmissions only if
the distance between the corresponding client machines is small,
and this is a~subject of our future research. Such an~extension of
the network topology was proposed by Bordenave, Foss and Shneer in~\cite{BFS}, where impatience of customers is not taken into account,
though. The authors study fluid limits in order to find out
whether the stochastic model is stable or not. The description of
fluid limits for our model with impatience omitted coincides with
the description of fluid limits in~\cite{BFS} adapted to the
network topology we consider here.

The paper is organised as follows. In Section~\ref{sec:stoch_model}, we present a
detailed description of the stochastic model. In Section~\ref{sec:fluid_model}, we
define the fluid model solutions and formulate their properties. In Section~\ref{sec:fluid_limits}, we state our main result
concerning convergence of sequences of the fluid scaled workload
processes to fluid model solutions. Sections~\ref{sec:proof_fms_exist_unique},~\ref{sec:proof_equilibrium_fms} and~\ref{sec:proof_fluid_limits} contain
the proofs of the results stated in Sections~\ref{sec:fluid_model} and~\ref{sec:fluid_limits}. Namely, in
Section~\ref{sec:proof_fms_exist_unique}, we establish existence and uniqueness of a~fluid model
solution for any initial state. In Section~\ref{sec:proof_equilibrium_fms}, existence, uniqueness and global asymptotic
stability of an~equilibrium fluid model solution are shown.
In Section~\ref{sec:proof_fluid_limits}, we prove the result of Section~\ref{sec:fluid_limits}. In the remainder of this section, we list the notation we use
throughout the paper.

\paragraph{Notation} The standard sets are denoted as follows: the real line $\mathbb{R}=(-\infty,\infty)$, the non-negative half-line $\mathbb{R}_+=[0,\infty)$,
the non-negative integers $\mathbb{Z}_+=\{0,1,2,\ldots\}$. All
vectors in the paper are $K$-dimensional with integer or real coordinates. The coordinates of a vector are denoted by the same symbol as the vector with lower indices $1, \ldots, K$ added. If a~vector has superscripts, overlining, or tilde sign, they
remain in its coordinates. For example, $\overline{x}^{\,0} \in \mathbb{R}_+^K$, $\overline{x}^{\,0}=(\overline{x}^{\,0}_1,\ldots,\overline{x}^{\,0}_K).$ The supremum and $L_1$-norm in $\mathbb{R}^K$ are $\|x\|=\max_{1 \leq i \leq K}|x_i|$ and
$\|x\|_{1}=\sum_{i=1}^K |x_i|$, respectively. Vector inequalities hold coordinate-wise. The coordinate-wise product of vectors is $x \ast y= (x_1y_1,...,x_Ky_K)$.

The signs $\weaklimit$, $\equal^\text{d}$ and $\leq_{\text{st}}$ stand for weak convergence, equality in distribution and stochastic order, respectively. Recall
that, for real-valued r.v.'s $X$ and $Y$, $X {\leq_{\text{st}}} Y$ if ${\vec{P}\{ X>t \}} \leq {\vec{P}\{
Y>t \}}$ for any $t \in \mathbb{R}$. For an~r.v. $X$ having a~distribution $G$, we write $X \sim G$. The notations $\upPi(a)$, with $a\in(0,\infty)$, and
$B(m,x)$, with $m \in \mathbb{Z}_+$, $x
\in [0,1]$, are used for the Poisson distribution with parameter~$a$ and binomial distribution with parameters $m$, $x$. With $B(0,x)$ we mean the degenerate
distribution localised at~$0$.

The complement of an~event $E$ is denoted by
$\overline{E}$.

For a~function $f \colon S \to \mathbb{R}^K$, its supremum norm is $\| f \|_S = \max_{1 \leq i \leq K} \sup_{x \in
S} |f_i(x)|$. Let $\mathcal{D}$ be the space of right-continuous functions $f \colon [0,\infty) \to \mathbb{R}^K$ with finite left
limits. Endow this space with the Skorokhod $J_1$-topology. Let
$\mathcal{G}$ be the class of continuous functions $g \colon [0,\infty) \to \mathbb{R}^K_+$
such that $\| g(t) \|>0$ for all $t\neq0$.

We use the signs $:=$ and $=:$ to introduce a~definition in the body
of a~formula.

\section{Stochastic model description}\label{sec:stoch_model} 

This section contains a~detailed description of the stochastic
model under study. In particular, it derives the dynamic equation for the system workload. All stochastic primitives introduced here are defined on a common probability space $\left( \upOmega, \mathcal{F}, \vec{P} \right)$ with expectation operator $\vec{E}$. The characteristic function of an event is denoted by $\vec{I}$.

\paragraph{Stochastic assumptions and the protocol} We consider an~ALOHA-type service system with impatient customers.
The system includes a~waiting room where customers arrive to, and
a~server. Time is slotted, i.e. arrivals and abandonments may
occur only at time instants $n=1,2,\ldots$. Time slot~$n$ is the
time interval $[n,n+1)$. We assume that there are $K<\infty$
classes of customers.

The arrival process is denoted by $\left\{ A(n); \  n \geq 1
\right\}$, where $A(n)=\left(A_1(n),\ldots,A_K(n)\right)$
and $A_i(n)$ is the number of class~$i$ customers arriving at time~$n$. We suppose that $\left\{ A(n); \  n \geq 1 \right\}$ is
an i.i.d. sequence and that $\vec{E}A(1) =: \upLambda=\left( \upLambda_1,\ldots,\upLambda_K
\right) \in (0,\infty)^K$.  The coordinates of the vectors $A(n)$ are
allowed to be dependent.

Let $W_i(n)$ be the number of class~$i$ customers present in
the system at time $n \geq 0$. As can be seen from the further description, $W_i(n)$ coincides with the workload at time~$n$ due
to class~$i$ customers. Hence, $\left\{ W(n) = \left(W_1(n),\ldots,W_K(n)\right); \ n \geq 0
\right\}$ denotes the workload process.

Each customer brings a~packet that takes exactly a~single time slot to be transmitted to the server. He also sets a~deadline for transmission: once the
deadline expires, the customer leaves the system even if his
packet has not been transmitted yet. In this case, we say that the
customer has abandoned the system due to impatience. Patience times of class~$i$
customers have geometrical distribution with parameter
$p_i$ and take values greater than or equal to~$1$. By~$p$ we
denote the~vector having parameters $p_i$ as its coordinates,
$p=\left( p_1,\ldots,p_K \right)$. Patience times of different
customers are mutually independent.

We now describe how a~transmission occurs. At the beginning of
time slot~$n$, each customer, independently of the others, starts transmission with probability
$1 / \|W(n)\|_1$ (and does not with probability $1 - 1 / \|W(n)\|_1$). If there is one customer transmitting, then
the transmission is going to be successful. Otherwise a~collision
happens. At the end of the time slot, customers learn the result.
If a~customer has succeeded to send his packet, he immediately
leaves the system. If a~customer has failed and he is from class~$i$, then with probability~$p_i$ he leaves due to impatience, and
with probability~$1-p_i$ he stays in the system to try to retransmit his packet later.

Throughout the paper we
assume the following.

\begin{assumption}\label{ass:nontrivial_arrivals}
The input process is non-trivial, i.e. $\vec{P}\{ \|A(1)\|_1 \geq 2 \}>0$.
\end{assumption}

\begin{assumption}\label{ass:overload}
The mean amount of work arriving to the system per time slot
exceeds the stability threshold for the model with no impatience
of customers, i.e. $\| \upLambda \|_1 > e^{-1}.$
\end{assumption}

\begin{remark}
The results of the paper can be generalised onto the case
when patience times are not simply geometrical random variables but
finite mixtures of those. It suffices, for all $i  = 1,\ldots,K$, to split customer class~$i$ into $k_i$ new classes, where $k_i$ is the number of mixture
components in the patience time distribution for class~$i$ customers.
\end{remark}

\paragraph{Workload dynamics} The sequence $\left\{ W(n);\ n\geq0 \right\}$
forms a~time-homogeneous Markov Chain. Its dynamics can be
described as follows: given a~history $\left\{ W(m); \
m=0,\ldots,n \right\}$ up to time $n$ with
$W(n)=x$, we have, for $i=1,\ldots,K$,
\begin{equation}\label{eq:one_step_mc}
W_i(n+1) \equal^{\text{d}} x_i + A_i - T_i(x) - I_i(x),
\end{equation}
where
\begin{itemize}
\item[$\bullet$] $A  =\left(A_1,\ldots,A_K\right) \equal^{\text{d}} A(1)$;
\item[$\bullet$] $T_i(x)$ represents the number of class~$i$
customers who are present in the system at time~$n$ but will leave
at the end of time slot~$n$ because of a~successful transmission,
\[
T_i(x) = \mathbb{I}\left\{ B_i\left( x \right) = 1\right\}
\prod\limits_{j \neq i} \vec{I}\left\{ B_j\left(
x \right) = 0\right\},
\]
where $B_i({x}) \sim B\left(x_i, 1/ \| {x} \|_1\right)$ if ${x} \neq (0,\ldots,0)$ and $B_i(0,\ldots,0)=0$;
\item[$\bullet$] $I_i({x})$ represents the number of class~$i$
customers who are present in the system at time~$n$ but will leave
at the end of time slot~$n$ due to impatience rather than a
successful transmission; given
${x} - {T}({x}) = {y}$,
\[
I_i({x}) = \widetilde{B}_i({y}) \sim B\left( y_i,p_i \right);
\]
\item[$\bullet$] for any ${x},\,{y} \in \mathbb{Z}^K_{+}$, the
random vector ${A}$ and the binomial r.v.'s
$B_i\left({x}\right)$,
$\widetilde{B}_j\left( {y} \right)$, $i,j = 1,\ldots,K$, are
mutually independent.
\end{itemize}

\begin{remark}\label{rem:total_transmissions} The number $T_i({x})$ of
successful transmissions by class~$i$ customers at a~time slot may
take only values $0$ and $1$. Moreover, $\|{T}({x}) \|_1 \leq 1$.
\end{remark}

\section{Fluid model}\label{sec:fluid_model}
In the present section, we define a~fluid analogue of the
stochastic model described in the previous section. As time and space are appropriately normalised, we expect that the difference
equation~\eqref{eq:one_step_mc} can be approximated by a~differential equation where the rate of increase is due to arrival rates, and the rate of decrease due to service completions and abandonments. In the single class case, one may expect such a~differential equation to look like (we omit the class index) $z'(t) = \upLambda - e^{-1} - p
z(t)$, since the throughput of ALOHA is~$e^{-1}$. In the multiclass case, this naturally extends to $z'_i(t) = \upLambda_i - e^{-1} z_i(t) / \left(\sum_{j=1}^K
z_j(t)\right) - p_i z_i(t)$, $i = 1,\ldots,K$. This will be made rigorous in Section~\ref{sec:fluid_limits}. We now proceed more formally.

\begin{definition}\label{def:fluid_model}
For ${z}^0 \in \mathbb{R}^K_+$, a~solution to the integral equation
\begin{equation}\label{eq:int_fluid_model}
{z}(t) = {z}^0 + t {\upLambda} - e^{-1}\!\!
\int_0^t \!\! {m}({z}(s)) ds - {p
\ast} \!\! \int_0^t \!\! {z}(s)ds,
\end{equation}
that comes from $\mathcal{G}$, i.e. that is continuous, non-negative and such that, for all $t \neq 0$, $\|{z}(t)\|>0$, is
called a~\textit{fluid model solution (FMS) with initial state
${z}^0$}. The function ${m} \colon \mathbb{R}^K_+ \to \mathbb{R}^K_+$ is given by
$$
{m}({x}) = \left\{
\begin{array}{rl}
{x} / \|{x}\|_1& \text{if } {x} \neq
(0,\ldots,0),\\
{\upLambda} / \|{\upLambda}\|_1& \text{if } {x} =
(0,\ldots,0).\\
\end{array}
\right.
$$
\end{definition}

\begin{remark}\label{rem:dif_fluid_model}
For a~function from $\mathcal{G}$, equation~\eqref{eq:int_fluid_model} is equivalent to the Cauchy problem
\begin{align}
{z}'(t) &= {\upLambda} - {p \ast z}(t)-e^{-1}{m}({z}(t)),  \label{eq:dif_fluid_model}\\
{z}(0) &= {z}^0, \nonumber
\end{align}
where \eqref{eq:dif_fluid_model} holds for $t \geq 0$ if ${z}^0 \neq (0,\ldots,0)$, and for $t > 0$ if ${z}^0 = (0,\ldots,0)$.
\end{remark}

\begin{remark} \label{rem:derivative_at_zero} Although ${m}(0,\ldots,0)$ does not appear in \eqref{eq:dif_fluid_model}, it needs to be defined for
further use in Section 7. We assign to ${m}(0,\ldots,0)$ the value of ${\upLambda} / \|{\upLambda}\|_1$, which is the limit of
${m}({\cdot})$ along FMS's trajectories as they approach $(0,\ldots,0)$. Indeed, the only point where a~fluid model
solution can take the value of $(0,\ldots,0)$ is $t = 0$. Let ${z}(\cdot)$ be
an FMS starting from ${z}(0) = (0,\ldots,0)$. For the moment suppose that
${z}(\cdot)$ is continuously differentiable at $t = 0$. Then~\eqref{eq:dif_fluid_model} and Taylor's expansion give, for small
$t > 0$,
\begin{equation}\label{eq:taylor}
{z}'(t) = {\upLambda} - {p \ast
z}(t) - e^{-1}\dfrac{t {z}'(0) + {o_K}(t)}{\sum_{i=1}^K
z'_i(0)t + o(t)},
\end{equation}
where $o(t) \in \mathbb{R}$, ${o_K}(t) \in \mathbb{R}^K$, and
$o(t)/t \to 0$, ${o_K}(t)/t \to (0,\ldots,0)$ as $t \to 0$. The
continuity of~${z}'(\cdot)$ at $t = 0$ and Assumption~\ref{ass:overload} guarantee that $\sum_{i=1}^K z'_i(0) > 0$,
so we pass to the limit as $t \to 0$ on both sides of~\eqref{eq:taylor} and get ${z}'(0)={\upLambda} - e^{-1}{z}'(0) / \left({\sum_{i=1}^K
z'_i(0)}\right)$. The last equation has a~unique solution
${z}'(0) = \left( 1-e^{-1} / \|{\upLambda} \|_1 \right) {\upLambda}$, which implies existence
of the limit $\lim_{t \to 0} {m}({z}(t))=
{\upLambda}/\| {\upLambda} \|_1$. Later on (see Section~\ref{sec:proof_fms_exist_unique}, Property~\ref{prty:derivative_at_zero}), we prove that, for
any function
${z}(\cdot)$ from $\mathcal{G}$ starting from
${z}(0) = (0,\ldots,0)$ and solving~(\ref{eq:dif_fluid_model}), there exists the derivative ${z}'(0) = \left( 1-e^{-1} / \|{\upLambda} \|_1
\right) {\upLambda}$, and hence
exists the limit $\lim_{t \to 0} {m}({z}(t))=
{\upLambda}/\| {\upLambda} \|_1$.
\end{remark}

In the remainder of the section, we discuss properties of FMS's.

\paragraph{Existence and uniqueness of FMS's} follow by the classical results from the theory of ordinary
differential equations if  the initial state is non-zero. Otherwise the proof is rather involved. The outline follows below; see Section~\ref{sec:proof_fms_exist_unique} for the full proof.

\begin{theorem}\label{th:fms_exist_unique}
For any initial state, a~fluid model solution exists and is
unique.
\end{theorem}

\textit{One-dimensional case.} Equation~\eqref{eq:int_fluid_model} turns into (we omit the class index) $ z(t) = z^0 + \left(\upLambda - e^{-1}\right)t
- p \int_0^t z(s) ds, $ and its unique solution is given by $z(t) = z^0
e^{-p\,t} + \left( (\upLambda - e^{-1})/{p} \right) \left( 1 - e^{-p\,t}
\right)$.

\textit{Multidimensional case, non-zero initial state.} Uniqueness follows easily by the Gronwall inequality (see e.g.~\cite[Chapter~III,
Paragraph~1]{Hartman}).

\begin{proposition}[Gronwall
inequality] \label{gronwall} Suppose that functions $u(\cdot)$ and $v(\cdot)$ are
non-negative and continuous in $[a,b]$, and that there exists a
constant $C \geq 0$ such that $ v(t)\leq C + \int_a^t v(s)u(s) ds$,
$a \leq t \leq b. $ Then $ v(t) \leq C \exp{\int_a^t u(s) ds}$, $a
\leq t \leq b. $ In particular, if $C = 0$, then $v(t) \equiv 0$.
\end{proposition}

Since the first order partial derivatives of the function
${m}({\cdot})$ are bounded on all sets
\begin{equation*}
\mathbb{R}^K_{\delta} := \left\{ {x}  \in 
\mathbb{R}^K_+ \colon \|{x}\|  \geq 
\delta \right\}, \quad \delta > 0,
\end{equation*}
this function is Lipschitz continuous on
all such sets. Let $c(\delta)$ be a~Lipschitz constant for
${m}({\cdot})$ on $\mathbb{R}^K_{\delta}$ with
respect to the supremum norm, i.e.
$\|{m}({x}) - {m}({x})\| \leq c(\delta) \|{x}-{y}\|$ for all ${x},
{y} \in \mathbb{R}^K_\delta$.

Suppose that ${z}(\cdot)$ and
$\widetilde{{z}}(\cdot)$ are two FMS's with the same non-zero initial state. They are continuous and non-zero at every point, and hence, for each $T > 0$,
there exists a~$\delta(T) > 0$ such that
${z}(t),\widetilde{{z}}(t) \in
\mathbb{R}^K_{\delta(T)}$, $0 \leq t \leq T$. We put
$\triangle{z}(\cdot)={z}(\cdot)-\widetilde{{z}}(\cdot)$
and $\triangle{m}({z})=
{m}({z})-{m}(\widetilde{{z}})$. Then $
-\triangle{z}(t)= e^{-1} \int_0^t
\triangle{m}({z}(s)) ds + {p \ast} \int_0^t
\triangle{z}(s) ds$, and, for $t \in [0,T]$, we have $\|\triangle{z}
\|_{[0,t]} \leq \left(e^{-1} c(\delta(T)) + \|{p}\| \right)
\int_0^t \|\triangle{z}\|_{[0,s]}ds.$ By the last inequality and the Gronwall inequality, we have $\|\triangle {z}
\|_{[0,t]} \leq 0$, $0\leq t \leq T$. Since $T$ is arbitrary, ${z}({\cdot})$ and $\widetilde{{z}}({\cdot})$ coincide on
$[0,\infty)$.

Existence of an~FMS with a~non-zero initial state
can be shown by the Peano existence theorem (the proof is postponed to Section~\ref{sec:proof_fms_exist_unique}).

\textit{Multidimensional case, zero initial state.} The techniques used in the previous case do not work here because they are based on the continuity properties of the function ${m}({\cdot})$
that fail as zero comes into play. So a~different
approach is required. We introduce a~family of integral equations depending on parameters $(\varepsilon,a) \in \mathbb{R}_+ \times
\mathbb{R}^K_+$ that includes (for $(\varepsilon,a)=(0,{p})$) an~equation equivalent to the Cauchy problem~\eqref{eq:dif_fluid_model}
with ${z}^0=(0,\ldots,0)$. We show that each equation of this family has a~solution. If ${\varepsilon > 0}$, then uniqueness of an
$(\varepsilon,a)$-solution is straightforward to show. In order to prove uniqueness of a~$(0,a)$-solution, we derive a~proper
estimate for it via solutions with other parameters. The whole
idea of this proof is adopted from~\cite{bez}.

\paragraph{Equilibrium FMS} We now discuss existence, uniqueness and asymptotic stability of a constant solution to the fluid model equation~\eqref{eq:dif_fluid_model}.
\begin{definition}
If a~differential equation ${z}'(t) =
{f} \bigl(t,{z}(t)\bigr)$, $t \geq 0$,
${z},{f} \in \mathbb{R}^K$, admits a~constant solution, then this solution is called an~\textit{equilibrium point}.
\end{definition}

By means of a Lyapunov function, we prove the following result; see Section~\ref{sec:proof_equilibrium_fms}.

\begin{theorem}\label{th:equilibrium_fms}
Suppose Assumption \ref{ass:overload} holds. Then there exists a~unique equilibrium point for the fluid model equation~\eqref{eq:dif_fluid_model}, which is given by
\begin{equation}\label{eq:equilibrium_formula}
z^{\textup{e}}_i= \dfrac{\upLambda_i}{x+p_i}, \quad i=1,\ldots,K, \quad \text{where $x$ solves} \quad \sum\limits_{i=1}^K \dfrac{p_i\upLambda_i}{x + p_i}=\|
{\upLambda} \|_1 - e^{-1},
\end{equation}
and any fluid model solution ${z}(t)$ converges to
${z}^{\textup{e}}$ as $t \to \infty$.
\end{theorem}

\section{Main theorem}\label{sec:fluid_limits}
In this section, we characterise the asymptotic behaviour under a~fluid scaling of the workload process of the stochastic model introduced in
Section~\ref{sec:stoch_model}, justifying the heuristics given in the previous section. First we specify the fluid scaling. Let $R$ be a~positive number.
Consider the stochastic model from Section~\ref{sec:stoch_model} with the impatience parameters~$p_i$ replaced by $p_i/R$, ${i \in \{1,\ldots,K\}}$. Denote the
workload process of the $R^{\text{th}}$ model by ${W}^R(\cdot)$, and scale it by $R$ both in space and time,
\begin{equation}\label{eq:scaling}
\overline{{W}}^R(t) := \dfrac{1}{R} {W}^R\left(\lfloor Rt \rfloor \right),
\quad t \geq 0.
\end{equation}
We refer to the processes~\eqref{eq:scaling} as the \textit{fluid scaled workload processes}. They take values in the Skorokhod space
$\mathcal{D}$.

\begin{remark}\label{rem:skorokhod}
If the limit is continuous (which is the case whenever we prove convergence in $\mathcal{D}$ in this paper), then convergence in the Skorokhod
$J_1$-topology is equivalent to uniform convergence on compact sets.
\end{remark}

Now we formulate the main theorem of the
paper and highlight the basic steps of the proof; the~detailed
proof will follow in Section~\ref{sec:proof_fluid_limits}.
\begin{theorem}\label{th:main}
Let Assumptions~\ref{ass:nontrivial_arrivals} and~\ref{ass:overload} hold. Then any sequence of the fluid
scaled workload processes $\overline{{W}}^R(\cdot)$ such
that $\overline{{W}}^R(0) \Rightarrow {z}^0 = const$, $R
\to \infty$, converges weakly in $\mathcal{D}$ to
the (unique by Theorem~\ref{th:fms_exist_unique}) FMS with initial state ${z}^0$.
\end{theorem}

The first step of the proof of Theorem~\ref{th:main} is to show that
the fluid scaled workload satisfies an~integral equation that differs
from the fluid model equation~\eqref{eq:int_fluid_model} by
negligible terms (for $R$ large enough). This appears to be the most difficult part of the proof. Then we obtain relative compactness of the family $\{ \overline{{W}}^R(\cdot); \ R > 0 \}$ and continuity of its weak limits by applying a~general result on relative compactness in $\mathcal{D}$. We also show
that the weak limits are bounded away from zero outside $t = 0$, which, together with their continuity, implies that they a.s. satisfy the fluid model
equation~\eqref{eq:int_fluid_model}, and hence all of them coincide a.s. with the unique FMS with initial state~${z}^0$.

\begin{remark}
In the literature on the fluid approximation approach, a~\textit{fluid limit} is any weak limit of the processes~\eqref{eq:scaling} along a~subsequence $R \to
\infty$. As it follows from Theorem~\ref{th:main}, for a~sequence of the processes~\eqref{eq:scaling} with $\overline{{W}}^R(0) \Rightarrow {z}^0 = const$, $R\to \infty$, there is a~unique fluid limit, and it is a~deterministic function, namely, the fluid model solution with initial state ${z}^0$.
\end{remark}

\section{Proof of Theorem~\ref{th:fms_exist_unique}} \label{sec:proof_fms_exist_unique}
We split the proof into two parts, for a~non-zero and zero initial state.

\subsection{Non-zero initial state}
Here we have to prove the existence result only, see Section~\ref{sec:fluid_model} for the proof of uniqueness. First we derive bounds for an~FMS using the
following lemma.
\begin{lemma}\label{lem:fms_bounds}
Let $S$ be either a~finite interval $[0,T]$ or the half-line
$[0,\infty)$, and let a~real-valued function $x(t)$ be continuous in
$S$ and differentiable in $S\setminus\{0\}$. Suppose that a~constant $C$ is such that $x(t) \geq C$ for $t \in S\setminus\{0\}$ implies $x'(t) \leq 0$. Then
$\sup_{t \in S} x(t) \leq \max\{x(0),C\}$.
\end{lemma}

\begin{proof} 
Let $\varepsilon > 0$. Suppose that $x(t) \in
\left(\max\left\{x(0), C \right\},\max\left\{x(0), C
\right\}+\varepsilon\right]$. Then, starting
from time $t$, the function $x(t)$ is decreasing at least until it
reaches level $C$. So $\sup_{t \in S} x(t) \leq
\max\left\{x(0), C \right\}+\varepsilon.$ Since
$\varepsilon
> 0$ is arbitrary, $\sup_{t \in S} x(t) \leq
\max\left\{x(0), C \right\}.$ \qed
\end{proof}

\paragraph{Bounds for a~fluid model solution} Let $S$ be either a~finite interval $[0,T]$ or the half-line $[0,\infty)$.
Suppose that a~function ${z}(\cdot)$ is non-negative with $\|{z}(\cdot)\|_1>0$ and continuous in $S$, and that it solves the
fluid model equation~\eqref{eq:dif_fluid_model} in
$S$. Then $\|{z}(t)\|_1=\sum_{i=1}^K
z_i(t)$ and the derivative $\|{z}(t)\|_1'$ exists for all
$t \in S$. Summing up the coordinates of equation~\eqref{eq:dif_fluid_model}, we get
\begin{equation*}
\|{\upLambda}\|_1 e^{-1}-p^{\ast}\|{z}(t)\|_1 \leq
\|{z}(t)\|_1' \leq
\|{\upLambda}\|_1-e^{-1}-p_{\ast}\|{z}(t)\|_1, \quad t
\in S,
\end{equation*}
where $p_{\ast}{=}\min_{1 \leq i \leq K} p_i$ and
$p^{\ast}{=}\max_{1 \leq i \leq K} p_i$. Then Lemma~\ref{lem:fms_bounds} applied to $x(\cdot)=\|{z}(\cdot)\|_1'$ and
$C=(\|{\upLambda}\|_1-e^{-1})/p_{\ast}$, and $x(\cdot)=-\|{z}(\cdot)\|_1'$ and $C=(\|{\upLambda}\|_1-e^{-1})/p^{\ast}$, implies that
\begin{alignat}{2}
	\sup\limits_{t \in S} \|{z}(t)\|_1 &\leq \max\left\{\|{z}(0)\|_1,
	\dfrac{\|{\upLambda}\|_1-e^{-1}}{p_{\ast}}\right\} & \, &=:u\left({z}(0)\right),
	\label{eq:upper_bound_norm} \displaybreak[0] \\
		\inf\limits_{t \in S} \|{z}(t)\|_1 &\geq \min\left\{\|{z}(0)\|_1,
		\dfrac{\|{\upLambda}\|_1-e^{-1}}{p^{\ast}}\right\} & &=:l\left({z}(0)\right).
		\label{eq:lower_bound_norm}
\end{alignat}

By Assumption~\ref{ass:overload}, we have
$u\left({z}(0)\right)>0$ for any non-negative ${z}(0)$,  and $l\left({z}(0)\right)>0$ for any non-negative and non-zero
${z}(0)$.

Inequality~\eqref{eq:upper_bound_norm} implies that $z_i'(t) \leq \upLambda_i -
e^{-1} u^{-1} \left( {z}(0) \right) z_i(t) - p_i z_i(t)
$, $i = 1,\ldots,K$. Then, by Lemma~\ref{lem:fms_bounds}, we get
\begin{equation}\label{eq:upper_bound_coordinates}
\sup\limits_{t \in S} {z}_i(t) \leq \max\left\{ z_i(0),
\dfrac{\upLambda_i}{e^{-1}u^{-1}\left({z}(0)\right)+p_i}
\right\}=:u_i\left({z}(0)\right).
\end{equation}

Similarly, if ${z}(0) \neq (0,\ldots,0)$, then inequality~\eqref{eq:lower_bound_norm} and Lemma~\ref{lem:fms_bounds} yield, for $i=1,\ldots,K$,
\begin{equation}\label{eq:lower_bound_coordinates}
\inf\limits_{t \in S} {z}_i(t) \geq \min\left\{ z_i(0),
\dfrac{\upLambda_i}{e^{-1}l^{-1}\left({z}(0)\right)+p_i}
\right\}=:l_i\left({z}(0)\right).
\end{equation}

\begin{remark}\label{rem:nonzero_outside_zero}
If ${z}(0) \neq (0,\ldots,0)$, then, by the bound~\eqref{eq:lower_bound_coordinates} and since
$z'_i(0)=\upLambda_i>0$ if ${z_i(0) = 0}$, we have
$$
\text{for all } \delta> 0, \quad \inf\limits_{t \in S, t \geq
\delta} \min\limits_{1 \leq i \leq K} {z}_i(t) > 0.
$$
\end{remark}

\paragraph{Existence of a~fluid model solution with a~non-zero
initial state} The key tool used in this proof is
the Peano existence theorem (see e.g.~\cite[Chapter~II,
Paragraph~2]{Hartman}).

\begin{proposition}[Peano] 
Suppose that a~function ${f} \colon \mathbb{R} \times \mathbb{R}^K \to
\mathbb{R}^K$ is continuous in the rectangle $B=\left\{
(t,{y}) \colon t_0 \leq t \leq t_0 + a,
\max_{1 \leq i \leq K} |y_i - y^0_i| \leq b \right\}.$ Let  $M \geq \|{f}\|_{\text{B}}$ and ${\alpha = \min\{a, b/M\}}$.
Then the Cauchy problem
\begin{align*}
{y}'(t) &= {f}\bigl(t,{y}(t)\bigr),\\
{y}(t_0) &= {y}^0
\end{align*}
has a~solution in the interval $[t_0, t_0 + \alpha]$ such that
$y(t) \in \text{B}$.
\end{proposition}

First note that it suffices to show existence of a~non-negative solution to equation~\eqref{eq:dif_fluid_model}. By
Remark~\ref{rem:nonzero_outside_zero}, it will not hit zero at $t > 0$, and hence will be an~FMS.

Further note that it suffices to consider initial states with all coordinates positive. Indeed, if ${z}(0)
$ is non-zero, there exists a~rectangle
$B=\left\{ \max_{1 \leq i \leq K}|z_i - z_i(0)| \leq b \right\}$ that does not
contain zero, and, consequently, the mapping
${f}({\cdot}) = e^{-1} {m}({\cdot}) +
{p}$ is continuous in $B$. Let $M =
\|{f}\|_{B}$ and $\alpha = b/M$. By the Peano theorem, there exists a~solution to equation~\eqref{eq:dif_fluid_model} in the interval $[0, \alpha]$. If
$z_i(0) > 0$, then, by continuity of ${z}(\cdot)$, there exists a~$t_i
\leq \alpha$ such that $
z_i(t) \geq z_i(0)/2$, $t \leq t_i$. If $z_i(0) = 0$, then $z'_i(0) = \upLambda_i>0$, and there exists a
$t_i \leq \alpha$ such that
$z_i(t) = \upLambda_i t(1 + o(1))
\geq \upLambda_i t / 2$, $t \leq t_i$. Therefore, with $\beta = \min_{1 \leq i \leq K} t_i$, we have $\inf_{t \leq \beta} \|{z}(t)\|_1 > 0$ and $z_i(\beta)> 0$, $1 \leq i \leq K$.

Suppose now that $z_i(0) > 0$, $1 \leq i \leq K$. We show that there exists a~constant $\alpha^\ast>0$ such that any non-negative solution
${z}^{(T)}(\cdot)$ to~\eqref{eq:dif_fluid_model} that is defined in an~interval $[0,T]$ can be continued onto $[0,T+\alpha^\ast]$ remaining non-negative ($\alpha^\ast$ does not depend on $T$ and ${z}^{(T)}(\cdot)$). This will complete the proof. Define the rectangle
$$B^{\ast} = \bigl\{ 0 <
l_i\left({z}(0)\right) / 2 \leq z_i \leq
u_i\left({z}(0)\right) + 
l_i\left({z}(0)\right) / 2, \; 1 \leq i \leq K \bigr\}$$ and the constants
$$M^{\ast} = \| {f} \|_{{B}^{\ast}}, \quad b^{\ast} =
\min\limits_{1 \leq i \leq k}
l_i\left({z}(0)\right) / 2, \quad \alpha^{\ast} = b^{\ast} /
M^{\ast}.$$
Consider $T=0$. Let ${B}_0 = \left\{\max_{1 \leq i \leq K} |z_i -
z_i(0)| \leq b^\ast\right\}$, then $M^{\ast} \geq \| {f}
\|_{{B}_0}$ because ${B}_0 \subseteq {B}^\ast$. By the Peano theorem, there exists a~solution to equation~\eqref{eq:dif_fluid_model} in the interval $[0,
\alpha^\ast]$, and it is non-negative because ${B}_0 \subseteq \mathbb{R}^K_+$. Consider $T>0$ and a~non-negative solution ${z}^{(T)}(\cdot)$ to~\eqref{eq:dif_fluid_model}
defined in $[0,T]$. By the bounds~\eqref{eq:upper_bound_coordinates} and~\eqref{eq:lower_bound_coordinates}, we have $l_i\left({z}(0)\right) \leq
z_i^{(T)}(T) \leq
u_i\left({z}(0)\right)$, $i = 1, \ldots, K$. Let ${B}_T = \{ \max_{1 \leq i
\leq K} |z_i - z_i^{(T)}(T)| \leq b^\ast \}$, then $M^{\ast} \geq \| {f} \|_{{B}_T}$ because ${B}_T \subseteq {B}^\ast$. By the Peano theorem, in the
interval $[0, \alpha^\ast]$, there exists a~solution ${y}^{(T)}(\cdot)$ to the Cauchy problem
\begin{align*}
{y}'(t) &= {\upLambda} -
{p \ast} {y}(t) -
e^{-1} {m}({y}(t)),\\
{y}(0)&={z}^{(T)}(T),
\end{align*}
and it is non-negative because 
${B}_T \subseteq \mathbb{R}^K_+$. Then 
$$
{z}^{(T+\alpha^\ast)}(t):=\left\{
\begin{array}{ll}
{z}^{(T)}(t),& t \in [0,T],\\
{y}^{(T)}(t-T),& t \in [T,T+\alpha^\ast]
\end{array}
\right.
$$
is a~non-negative solution to~\eqref{eq:dif_fluid_model} in $[0,T+\alpha^\ast]$.

\subsection{Zero initial state}
This proof is based on the ideas from~\cite{bez}. We introduce a~family of auxiliary integral equations parametrised by $(\varepsilon, a) \in
\mathbb{R}_+ \times \mathbb{R}_+^K$. By further Lemma~\ref{lem:equivalent_description_fluid_model}, the equation with parameters $(0,{p})$ is equivalent
to the fluid model equation with zero initial condition. By Property~\ref{prty:exist_unique_monotonicity}, each auxiliary equation has a~solution. By
Property~\ref{prty:fms_starting_zero_unique}, for any $a \in \mathbb{R}_+^K$, a~solution to the equation with parameters $(0,a)$
is unique.

\begin{lemma}[Equivalent description of the fluid model] \label{lem:equivalent_description_fluid_model}
For any initial state ${z}^0$, the set of fluid model solutions coincides with the set of functions from $\mathcal{G}$ that
solve the following system of integral equations: for $i = 1,\ldots, K$, $t \geq 0$,
\begin{equation}\label{eq:equivalent_description_fluid_model}
\begin{split}
{z}_i(t) =& \ z_i^0\exp\Biggl( - p_i t - \!\! \int_0^t \!\!
\dfrac{e^{-1}}{\|{z}(s)\|_1} \, ds \Biggr) \\
&+ \upLambda_i
\!\! \int_0^t \!\! \exp \Biggl(
- p_i(t - s) - \!\! \int_s^t \!\!
\dfrac{e^{-1}}{\|{z}(x)\|_1} \, dx \Biggr) ds.
\end{split}
\end{equation}
\end{lemma}

\begin{proof} As we differentiate equations~\eqref{eq:equivalent_description_fluid_model}, the fluid model equation~\eqref{eq:dif_fluid_model} follows. We
now show that~\eqref{eq:dif_fluid_model} implies~\eqref{eq:equivalent_description_fluid_model}. Let ${z}(\cdot)$ be an~FMS with initial state
${z}^0$ and consider the following Cauchy problem with respect to ${u}(\cdot)$: for $i=1,\ldots,K$,
\begin{equation}\label{eq:to_show_equivalence}
\begin{split}
u'_i(t) &= \upLambda_i - \left(p_i + \dfrac{e^{-1}}{\|{z}(t)\|_1}
\right) u_i(t), \quad t>0,\\
u_i(0) &= z_i^0.
\end{split}
\end{equation}
If~\eqref{eq:to_show_equivalence} has a~continuous solution, it must be unique. Indeed, suppose that ${u}(\cdot)$, $\widetilde{{u}}(\cdot)$ are
two continuous solutions to~\eqref{eq:to_show_equivalence}. Let ${w}(\cdot)={u}(\cdot) - \widetilde{{u}}(\cdot)$. Then, for $i=1,\ldots,K$,\begin{align*}
w'_i(t) &= - \left(p_i + \dfrac{e^{-1}}{\|{z}(t)\|_1}
\right) w_i(t), \quad t>0,\\
w_i(0) &= 0.
\end{align*}
Lemma~\ref{lem:fms_bounds} applied to $x(\cdot) = w_i(\cdot)$ and $C=0$, and $x(\cdot) = - w_i(\cdot)$ and $C=0$, $i=1,\ldots,K$, implies that
${w}(\cdot) \equiv {0}$.

Finally, ${z}(\cdot)$ is a~solution to~\eqref{eq:to_show_equivalence} by~\eqref{eq:dif_fluid_model}. For $i=1,\ldots,K$, $t \geq 0$, put $f_i(t)$ to be
the RHS of~\eqref{eq:equivalent_description_fluid_model}. Differentiating of ${f}(\cdot)$ implies that it is a~solution to~\eqref{eq:to_show_equivalence}, too. Since~\eqref{eq:to_show_equivalence} has a unique continuous solution, ${z}(\cdot)$ and ${f}(\cdot)$ coincide, and hence ${z}(\cdot)$ is a~solution
to~\eqref{eq:equivalent_description_fluid_model}. \qed
\end{proof}

\paragraph{Auxiliary fluid model solutions and their properties}

For each $(\varepsilon,a) \in \mathbb{R}_+ \times \mathbb{R}^K_+$, introduce the operator
${F}^{(\varepsilon,\,a)} \colon \mathcal{G} \to
\mathcal{G}$ defined by
\begin{equation*}
\begin{split}
F_i^{(\varepsilon,\,a)}({u})(t)=\varepsilon+\upLambda_i
\!\! \int_0^t \!\! \exp \Biggl( - a_i (t - s) - \!\! \int_s^t \!\!
\dfrac{e^{-1}}{\|{u}(x)\|_1}\,dx \Biggr) ds, \\
t \geq 0, \quad i=1, \ldots, K.
\end{split}
\end{equation*}

\begin{definition}\label{def:auxiliary_fms}
Let $(\varepsilon,a) \in \mathbb{R}_+\times
\mathbb{R}^K_+$. A~function ${z}(\cdot) \in \mathcal{G}$ solving
the integral equation
\begin{equation}\label{eq:auxiliary_fms}
{z}(t) = {F}^{(\varepsilon,\, a)}({z})(t),
\quad t \geq 0,
\end{equation}
is called an~\textit{$(\varepsilon,a)$-fluid model
solution} (for short, we write ``{\it $(\varepsilon, a)$-FMS\,}'').
\end{definition}

Further we establish a~number of properties of the auxiliary fluid
model solutions defined above, including existence and uniqueness of an~$(\varepsilon,a)$-FMS for any $(\varepsilon,a) \in
\mathbb{R}_+\times
\mathbb{R}^K_+$.

For $a \in \mathbb{R}^K_+$, put 
\begin{align*}
a^\ast &= \max_{1 \leq i \leq K} a_i, & a_\ast & = \min_{1 \leq i \leq K} a_i,\\
a^{\text{u}} & = (a^\ast,\ldots,a^\ast),& a^{\text{l}} & = (a_\ast,\ldots,a_\ast).
\end{align*}

\begin{prty}\label{prty:exist_unique_monotonicity} In what follows, ${z}^{(\varepsilon,\, a)} (\cdot)$ denotes an
$(\varepsilon,a)$-FMS.
\begin{itemize}
\item[\textup{(i)}] For any $(\varepsilon,a) \in \mathbb{R}_+\times
\mathbb{R}^K_+$, there exists an~$(\varepsilon,\, a)$-FMS.
\item[\textup{(ii)}] If $\varepsilon > 0$, then an~$(\varepsilon,a)$-FMS is unique.
\item[\textup{(iii)}] If $a_1=\ldots=a_K$, then a~$(0, a)$-FMS is unique and given by
\begin{equation} \label{eq:fms_equal_pi}
z_i^{(0,\,a)}(t)= \left\{
\begin{array}{ll}
\dfrac{\upLambda_i}{\| {\upLambda}
\|_1}\left(\|{\upLambda}\|_1-e^{-1}\right)t& \text{if } a_1=0,\\
\dfrac{\upLambda_i}{\| {\upLambda}
\|_1}\dfrac{\|{\upLambda}\|_1-e^{-1}}{a_1}(1-e^{-a_1 t})&
\text{if } a_1>0,
\end{array}
\right.
\end{equation}
\item[\textup{(iv)}] If $\varepsilon > \delta \geq 0$, then ${z}^{(\varepsilon,\, a)} (t) \geq {z}^{(\delta,\, a)} (t)$, $t \geq 0$.
\item[\textup{(v)}] A~$(0,a)$-FMS admits the bounds ${z}^{(0,\, a^{\textup{u}})} (t) \leq {z}^{(0,\, a)} (t) \leq
{z}^{(0,\, a^{\textup{l}})} (t)$, $t \geq 0$.
\end{itemize}
\end{prty}

\begin{proof} \textit{\textup{(i)} for $\varepsilon>0$.}
Put ${z}^0(\cdot) \equiv (\varepsilon,\ldots,\varepsilon)$ and ${z}^{n+1}(\cdot)={F}^{(\varepsilon,\, x)}({z}^n)(\cdot)$, $n \geq 0$.
Then ${z}^1(t) \geq (\varepsilon, \ldots, \varepsilon) = {z}^0(t)$, $t \geq 0$.
The operator ${F}^{(\varepsilon,\, a)}$ is monotone, that is, ${u}(t) \geq {v}(t)$ for all $t
\geq 0$ implies ${F}^{(\varepsilon,\, a)}({u})(t) \geq
{F}^{(\varepsilon, \, a)}({v})(t)$ for all $t
\geq 0$. Also ${F}^{(\varepsilon,\, a)}({u})(t) \leq (\varepsilon, \ldots, \varepsilon) +
t {\upLambda}$ for all ${u}(\cdot) \in \mathcal{G}$ and all $t \geq 0$. Hence, for each $t \geq 0$, the sequence
$\{{z}^n(t); \ n \geq 0 \}$ is non-decreasing and bounded from
above, and there exists the point-wise limit of
${z}^n(\cdot)$ as $n \to \infty$; denote it by ${z}(\cdot)$. Now we show that ${z}(\cdot)$ is an
$(\varepsilon,a)$-FMS. Take an~arbitrary
$t \geq 0$. For all $n \geq 0$, we have ${z}^{n}(\cdot) \geq (\varepsilon, \ldots, \varepsilon)$, and hence
$1 / \|{z}^{n}(\cdot)\|_1 \leq
1 / (K\varepsilon)$. Then, by the dominated convergence theorem, for all $s \in [0,t]$, $\int_s^t e^{-1} / \| {z}^{n}(x) \|_1 dx \to \int_s^t e^{-1} / \|{z}(x) \|_1 dx$
as $n \to \infty$. The last argument, together with $\exp \bigl( - a_i {(t - s)} - \int_s^t e^{-1} / \|{z}^{n}(x)\|_1 dx \bigr) \leq 1$, $s \in [0,t]$, and, again, the dominated convergence theorem, implies that ${F}^{(\varepsilon,\,a)}({z}^{n})(t) \to
{F}^{(\varepsilon,\, a)}({z})(t)$ as $n \to \infty$. So, indeed, ${z}(\cdot)$ satisfies
equation~\eqref{eq:auxiliary_fms} for all $t \geq 0$.

(ii) Let ${z}(\cdot)$ and $\widetilde{{z}}(\cdot)$ be two $(\varepsilon,a)$-FMS's. Take an~arbitrary $T > 0$. Since $ {a_i (t -s)} + \int_s^t e^{-1} / \|{z}(x)\|_1 dx \leq (\| a \| + (e K \varepsilon)^{-1})T$, $0 \leq s \leq t \leq T$, and the same is true for
$\widetilde{{z}}(\cdot)$, by Lipschitz continuity of $\exp(\cdot)$ on compact sets, there exists a~constant $\alpha(T)$ such that, for~$t \leq T$,
\begin{equation*}
\begin{split}
\|{z}(t) - \widetilde{{z}}(t)\| & \leq \alpha(T) e^{-1} \!\! \int_0^t
\!\! \int_s^t \Bigl| 1 / \| {z}(x) \|_1 - 
1 / \| \widetilde{{z}}(x) \|_1 \Bigr| dx ds \\
& \leq \alpha(T) e^{-1} T \!\! \int_0^t \Bigl| 1 / \| {z}(x)
\|_1 - 1 / \| \widetilde{{z}}(x) \|_1 \Bigr| dx.
\end{split}
\end{equation*}
Then, by Lipschitz continuity of $1 / \| {\cdot} \|_1$ on the set $\mathbb{R}_\varepsilon^K$ and by the Gronwall inequality, ${z}(\cdot)$ and
$\widetilde{{z}}(\cdot)$ must coincide in all finite intervals $[0,T]$, and hence they coincide on $[0,\infty)$.
 
(iii) Due to  Lemma~\ref{lem:equivalent_description_fluid_model}, $(0,a)$-FMS's are defined by the Cauchy problem
\begin{align*}
{z}'(t) &= {\upLambda} - a_1{z}(t)-e^{-1} {z}(t) / \| {z}(t) \|_1, \quad t>0, \\
{z}(0) &= {0}.
\end{align*}
Summing up its coordinates, we get the Cauchy problem
\begin{align*}
\|{z}(t)\|_1' &= (\| {\upLambda} \|_1 - e^{-1}) - a_1
\|{z}(t)\|_1, \quad t>0,\\
\|{z}(0)\|_1 &= 0,
\end{align*}
which admits a~unique solution
$$
\|{z}(t)\|_1= \left\{
\begin{array}{ll}
\left(\|{\upLambda}\|_1-e^{-1}\right)t,& \text{if } a_1=0,\\
(\|{\upLambda}\|_1-e^{-1})(1-e^{-a_1t}) / a_1,& \text{if }
a_1>0.
\end{array}
\right.
$$
In the case of $\varepsilon=0$ and $a_1=\ldots=a_K$, equation~\eqref{eq:auxiliary_fms} implies that ${z}(\cdot) / \| {z}(\cdot) \|_1 \equiv
{\upLambda} / \| {\upLambda} \|_1$. Then the~unique $(0,a)$-FMS is given by~\eqref{eq:fms_equal_pi}.

\textit{\textup{(i)} for $\varepsilon=0$, and \textup{(v)}.} In order to prove existence, consider the sequence
${z}^0(\cdot):={z}^{(0,\, a^{\text{u}})}(\cdot)$, ${z}^{n+1}(\cdot):={F}^{(0, \, a)}({z}^n)(\cdot)$,
$n \geq 0$. By the reasoning analogous to that in the case of $\varepsilon > 0$, the point-wise limit of this sequence is a~$(0,a)$-FMS. To prove the lower bound, consider the sequence ${z}^0(\cdot):={z}^{(0,a)}(\cdot)$,
${z}^{n+1}(\cdot):={F}^{(0,a^{\text{u}})}({z}^n)(\cdot)$, $n \geq 0$. It is non-increasing in $n$, and its point-wise limit
is the $(0,a^{\text{u}})$-FMS. Then ${z}^{(0,\, a^{\text{u}})}(t) \leq {z}^{0}(t)={z}^{(0,\, a)}(t)$ for all $t \geq 0$. Similarly, the upper bound holds.

(iv) Consider the sequence ${z}^0(\cdot) :={z}^{(\delta, \, a)}(\cdot)$,
${z}^{n+1}(\cdot):={F}^{(\varepsilon, \,a)}({z}^n)(\cdot)$, $n \geq 0$. It is non-decreasing in $n$, and its point-wise limit is the $(\varepsilon,a)$-FMS. Then ${z}^{(\varepsilon, \, a)}(t) \geq
{z}^{0}(t)={z}^{(\delta, \, a)}(t)$ for all $t \geq 0$. \qed
\end{proof}

We proceed with properties of $(0,a)$-FMS's at $t=0$ (\textit{cf.} Remark~\ref{rem:derivative_at_zero}).
\begin{prty}\label{prty:derivative_at_zero}
For any $(0,a)$-FMS ${z}^{(0, \, a)}(\cdot)$, its right derivative at $t=0$ is
well defined and
$( {z}^{(0, \, a)} )'(0)=( 1 - e^{-1} / \|{\upLambda}\|_1 ) {\upLambda}$. Also the limit
$\lim_{t \to 0} {z}^{(0, \, a)}(t) / \|
{z}^{(0, \, a)}(t) \|_1 = {\upLambda} / \| {\upLambda} \|_1$ exists.
\end{prty}

\begin{proof} Here are three possibilities: either $a^\ast \geq a_\ast>0$ or $a^\ast >
a_\ast = 0$, or $a^\ast = a_\ast = 0$. We~prove the property in the first case, the other two
cases can be treated similarly. By Property~\ref{prty:exist_unique_monotonicity},~(iii) and~(v), for $i = 1, \ldots, K$, $t \geq 0$,
\[
\dfrac{\upLambda_i}{\| {\upLambda}
\|_1}\dfrac{\|{\upLambda}\|_1-e^{-1}}{a^\ast}(1-e^{-a^\ast
t}) \leq z^{(0, \, a)}_i(t) \leq  \dfrac{\upLambda_i}{\|
{\upLambda}
\|_1}\dfrac{\|{\upLambda}\|_1-e^{-1}}{a_\ast}(1-e^{-a_\ast
t}).
\]
Then, for any sequence $t_n \to 0$, $n \to \infty$,
\[
\limsup\limits_{n \to \infty}
\dfrac{z_i^{(0, \, a)}(t_n)}{t_n}  \leq \dfrac{\upLambda_i}{\|
{\upLambda} \|_1 }\left( \|{\upLambda}\|_1-e^{-1}
\right) \lim\limits_{n \to \infty} \dfrac{1-e^{-a_\ast
t_n}}{a_\ast t_n}=\dfrac{\upLambda_i}{\| {\upLambda} \|_1
}\left( \|{\upLambda}\|_1-e^{-1} \right).
\]
Similarly, $\liminf_{n \to \infty} z_i^{(0, \, a)}(t_n) / t_n \geq (\upLambda_i / \| {\upLambda} \|_1 ) \left( \|{\upLambda}\|_1-e^{-1}
\right).$ Hence, the derivative exists. The~second result follows from Taylor's expansion, as was discussed in Remark~\ref{rem:derivative_at_zero}. \qed \end{proof}

Finally, we show uniqueness of a~$(0, a)$-FMS  by estimating it via the auxiliary FMS's with other parameters.
\begin{prty} \label{prty:fms_starting_zero_unique}
Fix $a \in \mathbb{R}_+^K$. For short, ${z}^\varepsilon$ denotes an~$(\varepsilon,a)$-FMS.
\begin{itemize} 
\item[\textup{(i)}] For $\varepsilon > 0$ and the function $\varphi_\varepsilon(t):=\int_0^t K e^{-1} / \| {z}^\varepsilon(s) \|_1 ds$,
\begin{equation}\label{eq:bound_via_others}
\| {z}^\varepsilon(t) - {z}^0(t) \|_1 \leq \varepsilon \left( K + \| a \|_1 + \varphi_\varepsilon(t) \right), \quad t \geq 0.
\end{equation}
\item[\textup{(ii)}] If $\varepsilon > 0$, $\varepsilon \to 0$, then $ \varepsilon \varphi_\varepsilon(t) \to 0$ for any $t > 0$.
\item[\textup{(iii)}] A~$(0, a)$-FMS is unique.
\end{itemize}
\end{prty}

\begin{proof} (i) Let $\varepsilon \geq 0$. By differentiating equation~\eqref{eq:auxiliary_fms}, we get, for $i=1,\ldots,K$,
\begin{align*}
(z_i^\varepsilon)'(t) &= \upLambda_i - \left(a_i + \dfrac{e^{-1}}{\| {z}^\varepsilon(t) \|_1} \right) \left( z_i^\varepsilon(t) - \varepsilon \right) ,
\quad t > 0,\\
z_i^\varepsilon(0) &= \varepsilon.
\end{align*}
Then integrating over $[0,t]$ yields
\begin{equation*}
z_i^\varepsilon(t) - \varepsilon = \upLambda_i t - \int_0^t \!\! \left(a_i + \dfrac{e^{-1}}{\| {z}^\varepsilon(s) \|_1} \right) \left( z_i^\varepsilon(s)- \varepsilon \right) ds , \quad t \geq 0,
\end{equation*}
which, after taking the sum in all coordinates, can be rewritten as
\begin{equation*}
\sum_{i=1}^K z_i^\varepsilon(t) + \sum_{i=1}^K \int_0^t a_i z_i^\varepsilon(s) ds = (\| {\upLambda} \|_1 - e^{-1}) t + \varepsilon K + \varepsilon \|
a \|_1 t + \varepsilon \varphi_\varepsilon(t), \quad t \geq 0.
\end{equation*}
The last equation implies that, for $\varepsilon > 0$,
\begin{equation} \label{eq:equality_via_others}
\begin{split}
\sum_{i=1}^K \left(z_i^\varepsilon(t) - z_i^0(t) \right) + \sum_{i=1}^K \int_0^t a_i \left(z_i^\varepsilon(s) - z_i^0(s) \right) ds &\\
= \varepsilon K +
\varepsilon \| a \|_1 t + \varepsilon \varphi_\varepsilon(t)&, \quad t \geq 0.
\end{split}
\end{equation}
Due to Property~\ref{prty:exist_unique_monotonicity},~(v), for $\varepsilon > 0$, both sums in the LHS of~\eqref{eq:equality_via_others} have non-negative
summands. By omitting the second sum, we obtain the bound~\eqref{eq:bound_via_others}.

(ii) Suppose that $a^\ast > 0$ (the other case can be treated similarly). Property~\ref{prty:exist_unique_monotonicity},~(iv) and~(v), together with $\|
{z}^\varepsilon(\cdot) \|_1 \geq K \varepsilon$, implies that, for $\varepsilon > 0$,
\begin{equation*}
\varphi_\varepsilon(t) = \int_0^t \!\! \dfrac { K e^{-1} } { \| {z}^\varepsilon(s) \|_1} \, ds \leq
\int_0^t \!\! \dfrac{K e^{-1}}{ \max \{ K \varepsilon,\|
{z}^{(0, \, a^{\text{u}})}(s) \|_1 \}}\,ds.
\end{equation*}
By~\eqref{eq:fms_equal_pi}, $\|
{z}^{(0, \, a^{\text{u}})}(s) \|_1 \leq K \varepsilon$ if and only if $s \leq f(\varepsilon):= \dfrac{1}{a^\ast}\ln
\dfrac{\| {\upLambda} \|_1 - e^{-1}}{\| {\upLambda} \|_1
- e^{-1}-K \varepsilon a^\ast}$. We have $f(\varepsilon)>0$ for $\varepsilon$ small enough and $f(\varepsilon) \to 0$, $\varepsilon \to 0$. Put
$\beta=K e^{-1} / (\| {\upLambda} \|_1 - e^{-1})$, then
\begin{align*}
\varepsilon \varphi_\varepsilon{t} =& \ e^{-1} f(\varepsilon)+\varepsilon \beta a^\ast\!\!
\int_{f(\varepsilon)}^t \!\! \dfrac{1}{1-e^{-a^\ast
s}}\,ds\\ 
=& \ e^{-1} f(\varepsilon) + \varepsilon \beta a^\ast \biggl( t - f(\varepsilon) +\dfrac{1}{a^\ast}\ln\left(1-e^{-a^\ast t}\right) \biggr) \\
&- \dfrac{\varepsilon \beta}{f(\varepsilon)} \biggl( f(\varepsilon)\ln\left( 1- e^{-a^\ast f(\varepsilon)} \right) \biggr).
\end{align*}
In the very RHS of the last equation, convergence of the first two summands to $0$ as $\varepsilon \to 0$ is clear. The first multiplier of the last
summand tends to a~finite constant, and the second multiplier tends to 0.

(iii) Suppose that ${z}^0(\cdot)$ and $\widetilde{{z}}^{\! \ 0}(\cdot)$ are two $(0,\boldsymbol{\pi})$-FMS's. For any $t > 0$, by~(i) and~(ii), $\|
{z}^0(t) - \widetilde{{z}}^{\! \ 0}(t) \|_1 \leq \| {z}^0(t) - {z}^\varepsilon(t) \|_1 + \| {z}^\varepsilon(t) -
\widetilde{{z}}^{\! \ 0}(t) \|_1 \to 0$ as $\varepsilon \to 0$. Hence, ${z}^0(\cdot)$ and $\widetilde{{z}}^{\! \ 0}(\cdot)$ must coincide in $[0,\infty)$. \qed
\end{proof}

\section{Proof of Theorem \ref{th:equilibrium_fms}} \label{sec:proof_equilibrium_fms}
\textit{Existence and uniqueness.} The function $f(s)=\sum_{i=1}^K p_i \upLambda_i / (s+p_i)$ is continuous and strictly decreasing in $(0,\infty)$, and takes all values between $\|{\upLambda} \|_1$ and 0 as $s$ goes along $(0,\infty)$. Then, by Assumption~\ref{ass:overload}, there exists a~unique ${z}^{\text{e}}$ satisfying \eqref{eq:equilibrium_formula}, and all its coordinates are positive. Equilibrium FMS's are defined by the equation\begin{equation}\label{eq:equation_equilibrium}
{\upLambda} - {p} \ast {z}^{\text{e}} - e^{-1}
 {z}^{\text{e}} / \left\| {z}^{\text{e}}
\right\|_1=0.
\end{equation}
In order to prove the first part of the theorem, we have to check that~\eqref{eq:equilibrium_formula} is a~solution to~\eqref{eq:equation_equilibrium}, and
that, if there is a~solution to~\eqref{eq:equation_equilibrium}, then it is necessarily~\eqref{eq:equilibrium_formula}.

By plugging~\eqref{eq:equilibrium_formula} into~\eqref{eq:equation_equilibrium} multiplied by $\left\|
{z}^{\text{e}} \right\|_1$, we obtain, for $i=1,\ldots,K$,
\begin{align*}
\biggl(\upLambda_i - \dfrac{p_i \upLambda_i}{x+p_i}
\biggr)\sum\limits_{j=1}^K
\dfrac{\upLambda_j}{x+p_j} - \dfrac{e^{-1}\upLambda_i}{x+p_i} =
\dfrac{\upLambda_i x}{x+p_i}\sum\limits_{j=1}^K
\dfrac{\upLambda_j}{x+p_j} - \dfrac{e^{-1}\upLambda_i}{x+p_i}\\
= \dfrac{\upLambda_i}{x+p_i}\biggl(\sum\limits_{j=1}^K \dfrac{\upLambda_j x}{x+p_j} - e^{-1}
\biggr) =\dfrac{\upLambda_i}{x+p_i}\biggl(\sum\limits_{j=1}^K
\upLambda_j - \sum\limits_{j=1}^K \dfrac{\upLambda_j p_j}{x+p_j}
-e^{-1} \biggr) = 0.
\end{align*}
So, indeed,~\eqref{eq:equilibrium_formula} is an~equilibrium FMS.

Suppose now that ${z}^{\text{e}}$ is a~solution to~\eqref{eq:equation_equilibrium}. By solving coordinate~$i$ of~\eqref{eq:equation_equilibrium} with
respect to $z_i^{\text{e}}$, we get
\[
z_i^{\text{e}}=\dfrac{\upLambda_i}{p_i+e^{-1}/\| {z}^{\text{e}} \|_1}, \quad i=1,\ldots,K. 
\]
Plug the last relation into the sum of the coordinates of~\eqref{eq:equation_equilibrium}, then
\[
\sum_{i=1}^K \dfrac{p_i\upLambda_i}{p_i+e^{-1}/\| {z}^{\text{e}} \|_1} = \| \upLambda \|_1 - e^{-1},
\]
Hence, ${z}^{\text{e}}$ satisfies~\eqref{eq:equilibrium_formula} with $x=e^{-1}/\| {z}^{\text{e}} \|_1$.

\textit{Stability.} When proving the second part of the theorem, we refer to the following result (see~\cite[Chapter~XIV,
Paragraph~11]{Hartman}).

Consider an~autonomous system of differential equations
\begin{equation}\label{eq:autonomous system}
{z}'(t)={f}({z(t)}), \quad t \geq 0, \quad
{z}, {f} \in \mathbb{R}^K.
\end{equation}

\begin{proposition}\label{stability}
Suppose that the function 
${f}({z})$ is continuous in an open set $E \subseteq \mathbb{R}^K$, and that, for any
${z}(0) \in E$, there exists a~unique solution
${z}(\cdot)$ to system~\eqref{eq:autonomous system} and
${z}(t) \in E$ for all $t \geq 0$. Suppose also that there
exists a~non-negative continuously differentiable in $E$ function
$V({z})$ such that $V({z}) \to \infty$ as
$\|{z}\| \to \infty$, and the trajectory derivative of
$V({z})$ with respect to system~\eqref{eq:autonomous system},
defined by $V'({z})= \sum_{i=1}^K
f_i({z}\!) \, \partial V /\partial z_i ({z}\!)$,
is non-positive. If there exists a~unique point
$\widetilde{{z}}$ where $V'(\widetilde{{z}}\,)=0$,
then the solution $\widetilde{{z}}(t)\equiv
\widetilde{{z}}$ of system~\eqref{eq:autonomous system} is
asymptotically stable in $E$, i.e. any solution to~\eqref{eq:autonomous system} with initial condition in $E$ converges
to $\widetilde{{z}}$ as $t \to \infty$.
\end{proposition}

By Remark~\ref{rem:nonzero_outside_zero} and Property~\ref{prty:exist_unique_monotonicity},~(iii) and~(v), any FMS at any time $t>0$ has all coordinates
positive. Then it suffices to show convergence to the equilibrium point for FMS's that start in the interior of $\mathbb{R}_+^K$. This, in turn, follows
from Proposition~\ref{stability} with the open set $E=(0,\infty)^K$, equation~\eqref{eq:dif_fluid_model} and the Lyapunov function
$$
V({z})=\sum\limits_{i=1}^K \dfrac{y_i^2}{z_i^{\text{e}}}, \quad \text{where } y_i:=z_i-z^{\text{e}}_i.
$$
By plugging $y_i$'s into~\eqref{eq:dif_fluid_model}, we get
$$
z'_i=-p_i y_i+e^{-1}\dfrac{z^{\text{e}}_i}{\|
{z}^{\text{e}} \|_1}-e^{-1}\dfrac{y_i+z^{\text{e}}_i}{\|
{y}+{z}^{\text{e}} \|_1}, \quad i=1,\ldots,K.
$$
Then,
\begin{equation*}
V'({z}) = \sum\limits_{i=1}^K \dfrac{2 y_i z_i'}{z^{\text{e}}_i} = -\sum\limits_{i=1}^K \dfrac{2
p_i y_i^2}{z^{\text{e}}_i} - \dfrac{2e^{-1}}{\|{y}+{z}^{\text{e}}\|_1} \biggl( \sum_{i=1}^{K}\dfrac{y_i^2}{z^{\text{e}}_i} -
\dfrac{\bigl(\sum_{i=1}^{K}
y_i\bigr)^2}{\sum_{i=1}^{K}
z^{\text{e}}_i} \biggr).
\end{equation*}

We have to check that $V'({z})$ is non-positive on $(0,\infty)^K$. For $k=1,\ldots,K$, let
\begin{equation*}
u_k({y}) = \sum_{i=1}^{k}\dfrac{y_i^2}{z^{\text{e}}_i} - \dfrac{\bigl(\sum_{i=1}^{k}
y_i\bigr)^2}{\sum_{i=1}^{k}
z^{\text{e}}_i}.
\end{equation*}
In particular,
\begin{equation} \label{eq:lyapunov_function}
V'({z}) = -\sum\limits_{i=1}^K \dfrac{2
p_i y_i^2}{z^{\text{e}}_i} - \dfrac{2e^{-1}}{\|{y}+{z}^{\text{e}}\|_1}\,u_K({y}).
\end{equation}
The quadratic form $u_K({y})$ is non-negative on the set $\{ {y} \colon {y} + {z}^{\text{e}} \in (0,\infty)^K \}$,
which we can shown by applying the Lagrange reduction to canonical form. We start from separating out the terms with $y_K$ and obtain
\begin{equation*}
u_K({y}) = \dfrac{1}{\sum_{i=1}^K
z^{\text{e}}_i} \biggl( \sqrt{\dfrac{z^{\text{e}}_K}{\sum_{i=1}^{K-1}
z^{\text{e}}_i}}\sum_{i=1}^{K-1}y_i-\sqrt{\dfrac{\sum_{i=1}^{K-1}
z^{\text{e}}_i}{z^{\text{e}}_K}}y_K
\biggr)^2 + 
u_{K-1}({y}).
\end{equation*}
Iterating the procedure for $y_{K-1}$,\ldots,$y_2$, we transform $u_K({y})$ into a~sum of squares. Then, by~\eqref{eq:lyapunov_function},
$V'({z})$ is non-positive on $(0,\infty)^K$. \qed

\section{Proof of Theorem~\ref{th:main}} \label{sec:proof_fluid_limits}
The proof is organised as follows. Section~\ref{ssec:representation} contains a~representation of the workload process before and after the fluid scaling. In Section~\ref{ssec:proof_theorem}, we formulate two auxiliary results (Lemmas~\ref{lem:G2_to_zero} and~\ref{lem:G3_martingale_to_zero}), and then proceed with
the proof of relative compactness of the family $\{ \overline{{W}}^R(\cdot); \ R>0 \}$ and weak convergence (as $R \to \infty$) along subsequences from this family. Proofs of Lemmas~\ref{lem:G2_to_zero} and~\ref{lem:G3_martingale_to_zero} are given in Sections~\ref{ssec:proof_lemma3} and~\ref{ssec:proof_lemma4},
respectively.

\subsection{Workload representation} \label{ssec:representation}
Throughout the proof, unless otherwise stated, we use the following representation of the
workload dynamics:
\begin{equation}\label{eq:wl_repr}
{W}^R(n+1) = {W}^R(n) + {A}(n+1) -
{T}\left(n,{W}^R(n)\right) -
{I}^R\left(n,{W}^R(n)\right),
\end{equation}
where, for ${x} \in \mathbb{Z}^K_+$ and $i=1,\ldots,K$,
\begin{align*}
p_i({x}) &= 
\left\{ \begin{array}{ll}
B\left(x_i, 1 / \|{x}\|_1 \right)(\{1\}) \prod_{j
\neq i} B\left( x_j, 1 / \|{x}\|_1 \right)(\{0\}) & \text{if } {x} \neq {0}, \\
0  & \text{if } {x} = {0}, \\
\end{array} \right.
\displaybreak[0] \\
T_i(n,{x}) &= \vec{I} \biggl\{
\sum\limits_{j=1}^{i-1}p_j({x}) \leq  U(n)
 < \sum\limits_{j=1}^{i}p_j({x}) \biggr\}, \displaybreak[0] \\
I^R_i(n,{x})&= \sum_{j=1}^{x_i-T_i(n,{x})}
\xi^R_i(n,j),
\end{align*}
and
\begin{itemize}
\item[$\bullet$] $\left\{U(n); \ n \geq 0\right\}$ is an~i.i.d. sequence, and $U(0)$ is distributed uniformly over $[0,1]$,
\item[$\bullet$] $\left\{\xi_i^R(n,j); \ j \geq 1\right\}$,  $n \geq 0$, $i=1,\ldots,K$, are independent i.i.d.
sequences of Bernoulli r.v.'s, and   $\vec{P}\{\xi_i^R(n,1)=1  \} = p_i/R = 1 - \vec{P}\{\xi_i^R(n,1)=0  \}$,
\item[$\bullet$] the sequences $\left\{{A}(n); \ n
\geq 1\right\}$, $\left\{U(n); \ n \geq 0\right\}$ and
$\left\{\xi_i^R(n,j); \ j \geq 1\right\}$, $n \geq 0$, $i=1,\ldots,K$,
are mutually independent and also do not depend on the initial
condition ${W}^R(0)$.
\end{itemize}
For short, we put
$$
h(x)=\left\{
\begin{array}{ll}
0,& x=0,\\
\left( 1-1/x \right)^{x-1},& x \geq 1.
\end{array}
\right.
$$
Then, in particular,
\begin{align*}
p_i({x}) &=
h\bigl(\|{x}\|_1\bigr) m_i\bigl({x}\bigr), \displaybreak[0] \\
\vec{E} \Bigl[{T}\left(i,{W}^R(i)\right) \Big|
{W}^R(i) \Bigr] &=
h \left( \|{W}^R(i)\|_1 \right) {m}\left( {W}^R(i) \right) , \displaybreak[0] \\
\vec{E} \Bigl[{I}^R\left(i,{W}^R(i)\right) \Big|
{W}^R(i) \Bigr] &= \dfrac{{p}}{R} \ast \Bigl(
{W}^R(i) - h \left( \|{W}^R(i)\|_1 \right)
{m}\left( {W}^R(i) \right)\Bigr).
\end{align*}

We now transform the workload dynamics into an~integral equation that, as we show later, differs from the fluid model equation~\eqref{eq:int_fluid_model} by
the terms that vanish as ${R \to \infty}$. For any $n \in \mathbb{Z}_+$, we have
\begin{align}
\begin{split}
{W}^R(n) &= {W}^R(0) + \sum\limits_{i=1}^n
{A}(i) - \sum\limits_{i=0}^{n-1}
{T}\left(i,{W}^R(i)\right) - \sum\limits_{i=0}^{n-1}
{I}^R\left(i,{W}^R(i)\right)
\end{split} \nonumber \\
\begin{split}
& =  {W}^R(0) + n {\upLambda}  -
\sum\limits_{i=0}^{n-1}h \left( \|{W}^R(i)\|_1 \right)
{m}\left( {W}^R(i) \right)  \\
& \quad - \dfrac{{p}}{R} \ast \sum\limits_{i=0}^{n-1}\Bigl({W}^R(i) - h \left( \|{W}^R(i)\|_1 \right)
{m}\left( {W}^R(i) \right)\Bigr) + {M}^R(n),
\end{split} \label{eq:workload}
\end{align}
where the sequence $\left\{ {M}^R(n); \ n \geq 0 \right\}$ forms a~zero-mean 
martingale since
\begin{align*}
{M}^R(n) & =  \sum\limits_{i=1}^n \biggl( {A}(i) -\vec{E} {A}(i) \biggr) -
\sum\limits_{i=0}^{n-1}  \biggl( {T}\left(i,{W}^R(i)\right) -\vec{E}
\Bigl[{T}\left(i,{W}^R(i)\right) \Big|
 {W}^R(i) \Bigr] \biggr)\\
& \quad -  \sum\limits_{i=0}^{n-1}
\biggl( {I}^R\left(i,{W}^R(i)\right) - \vec{E}
\Bigl[{I}^R\left(i,{W}^R(i)\right) \Big|
{W}^R(i) \Bigr] \biggr).
\end{align*}
Introduce the fluid scaled version of the martingale $\left\{ {M}^R(n); \ n \geq 0 \right\}$ analogous to that of the workload process:
$$
\overline{{W}}^R(t)=\dfrac{1}{R}{W}^R\bigl( \lfloor
Rt \rfloor \bigr), \quad
\overline{{M}}^R(t)=\dfrac{1}{R}{M}^R\bigl( \lfloor
Rt \rfloor \bigr),
$$
Then equation~\eqref{eq:workload} turns into the integral
equation
\begin{equation}\label{eq:scaled_workload}
\begin{split}
\overline{{W}}^R(t) =& \ \overline{{W}}^R(0) +
\dfrac{ \lfloor Rt \rfloor}{R} {\upLambda} \\ 
& - \biggl((1, \ldots, 1)-\dfrac{{p}}{R}\biggr) \ast \!\!\!\!\!
\int\limits_0^{\lfloor Rt \rfloor / R}\!\!\!\!\! h \left( R
\|\overline{{W}}^R(s)\|_1 \right) {m}\left(
\overline{{W}}^R(s) \right)\Bigr) ds \\
&  -  {p} \ast \!\!\!\!\!
\int\limits_0^{\lfloor Rt \rfloor / R} \!\!\!\!\!
\overline{{W}}^R(s)\,ds + \overline{{M}}^R(t).
\end{split}
\end{equation}

Finally, we rewrite equation~\eqref{eq:scaled_workload} as
\begin{equation}\label{eq:scaled_workload_final}
\overline{{W}}^R(t)  =  \overline{{W}}^R(0) + t
{\upLambda} - e^{-1}\!\!\! \int\limits_0^t \!\!
{m}\left(\overline{{W}}^R(s)\right) ds - {p \ast} \!\! \int\limits_0^t \!\!\overline{{W}}^R(s) ds + {G}^R(t),
\end{equation}
where
\begin{align*}
{G}^R(t) =& \ \overline{{M}}^R(t) + {G}^{1,R}(t) + {G}^{2,R}(t) + {G}^{3,R}(t), \displaybreak[0] \\
{G}^{1,R}(t) =& \ \left(\dfrac{\lfloor Rt \rfloor}{R} -
t\right){\upLambda }, \displaybreak[0] \\
{G}^{2,R}(t) =& \ e^{-1}\!\!\! \int\limits_0^t \!\!
{m}\left(\overline{{W}}^R(s)\right) ds \\
&-\biggl((1, \ldots, 1) -\dfrac{{p}}{R}\biggr){\ast}
\!\!\!\!\! \int\limits_0^{\lfloor Rt \rfloor / R} \!\!\!\!\!
h\left( R \|\overline{{W}}^R(s)\|_1
\right){m}\left(\overline{{W}}^R(s)\right) ds, \displaybreak[0] \\
{G}^{3,R}(t) =& \ {p \ast} \!\!\!\!\!
\int\limits_{\lfloor Rt \rfloor / R}^t \!\!\!\!\!
\overline{{W}}^R(s)\,ds.
\end{align*}

\subsection{Relative compactness and limiting equations} \label{ssec:proof_theorem}

We first discuss convergence ${G}^R(\cdot) \Rightarrow (0, \ldots, 0)$ as $R \to \infty$ in $\mathcal{D}$. By Remark~\ref{rem:skorokhod},
${G}^{1,R}( \cdot ) \Rightarrow (0,\ldots,0)$ as $R \to \infty$. Weak convergence to zero of the three other summands in ${G}^R(\cdot)$
follows from Lemmas~\ref{lem:G2_to_zero} and~\ref{lem:G3_martingale_to_zero}.

\begin{lemma} \label{lem:G2_to_zero} Let Assumptions~\ref{ass:nontrivial_arrivals} and~\ref{ass:overload} hold. Then
\begin{itemize} \itemsep0pt \parsep0pt 
\item[\textup{(i)}] for any $\delta > 0$ and $\varepsilon > 0$, there exists a
$\gamma=\gamma(\delta,\varepsilon)>0$ such that
\begin{equation*}
\liminf_{R \to \infty} \ \vec{P} \left\{ \varphi^R(\gamma R)
\leq \delta R \right\} \geq 1-\varepsilon,
\end{equation*}
where $\varphi^R(\gamma R) := \inf \{ n \geq 0 : W^R(n) \geq \gamma R \},$
\item[\textup{(ii)}] for any $\varepsilon > 0$ and $\upDelta > \delta > 0$,
there exists a~$C= C(\varepsilon,\delta,\upDelta) > 0$
such that
\begin{equation}\label{eq:suf_cond_G2_to_zero}
\liminf\limits_{R \to \infty} \ \vec{P} \Bigl\{ \inf_{\delta \leq t
\leq \upDelta} \! \|\overline{{W}}^R(t)\|_1 \geq
C \Bigr\} \geq 1-\varepsilon,
\end{equation}
\item[\textup{(iii)}] ${G}^{2,R}(\cdot) \Rightarrow (0,\ldots,0)$ in $\mathcal{D}$ as $R \to \infty$.
\end{itemize}
\end{lemma}

\begin{lemma} \label{lem:G3_martingale_to_zero}
Suppose $\overline{{W}}^R(0) \Rightarrow {z}^0$ as $R \to \infty$. Then $
{G}^{3,R}(\cdot) \Rightarrow (0,\ldots,0)$ and $\overline{{M}}^R(\cdot) \Rightarrow (0,\ldots,0)$ in $\mathcal{D}$ as $R \to \infty$.
\end{lemma}

Now we are in a position to prove the theorem. The proof consists of two steps. First, we establish relative compactness of a~sequence 
$\{\overline{W}^R(\cdot); \ {R \to \infty} \}$ such that $\overline{{W}}^R(0) \Rightarrow {z}^0$ (note that the assumption of ${z}^0$
being deterministic is not essential here, it is used in the next step). Second, we show that all weak limits of such a~sequence a.s. coincide with the FMS
starting from $ {z}^0 $, which means convergence of the sequence to the FMS starting from $ {z}^0 $.

\textit{Relative compactness.} We apply the following known result (see \cite[Chapter~3, Corollary~7.4 and Theorem~10.2(a)]{EthierKurtz}).
\begin{proposition} \label{tightness}
Let $\{ {X}^n(\cdot); \ n \geq 1 \}$ be a~sequence of processes with sample paths in $\mathcal{D}$. Suppose that, for any $\eta > 0$ and $t
\geq 0$, there exists a~compact set $\upGamma_{\eta,t} \subset \mathbb{R}^K$ such that
$$
\liminf_{n \to \infty} \vec{P} \{ {X}^n(t) \in \upGamma_{\eta,t} \} \geq 1 - \eta.
$$
Suppose further that, for any $\eta > 0$ and $T > 0$, there exists a~$\delta > 0$ such that
$$
\limsup_{n \to \infty} \vec{P} \{ \omega( {X}^n, \delta, T ) \geq \eta \} \leq \eta,
$$
where $\omega( {x}, \delta, T ):=\max_{1 \leq i \leq
K}\sup\{ |x_i(s)-x_i(t)| : s,t \in [0,T], |t-s|< \delta \}$. Then the sequence $\{ {X}^n(\cdot); \ n \geq 1\}$ is relatively compact and all its weak limit points are a.s. continuous.
\end{proposition}

Consider a~sequence $\{ \overline{W}^R(\cdot); \ {R \to \infty} \}$ such that $\overline{{W}}^R(0) \Rightarrow {z}^0$. By
\eqref{eq:wl_repr}, we have
$
{0} \leq \overline{{W}}^R(t) \leq \overline{{W}}^R(0) + \sum_{n=1}^{\lfloor Rt \rfloor} {A}(n)/R \Rightarrow {z}^0 + t
{\upLambda}
$,
and the first condition of the proposition holds. Now we check the second condition. Equation~\eqref{eq:scaled_workload_final} implies, for $t>s$,
$i=1,\ldots,K$,
\begin{equation*}
\begin{split}
\overline{W}^R_i(t)-\overline{W}^R_i(s)=& \ \upLambda_i(t-s)-e^{-1} \!\! \int_s^t \!\! m_i\left( \overline{{W}}^R(x)
\right) dx \\
&- p_i \!\! \int_s^t \!
\overline{W}_i^R (x) \,dx 
+ G_i^R(t) - G_i^R(s).
\end{split}
\end{equation*}
Since $m_i(\cdot) \leq 1$, we have, for $t,s \in [0,T]$, $|t-s|<\delta$,
\begin{equation}\label{eq:t-s}
|\overline{W}^R_i(t)-\overline{W}^R_i(s)|\leq\upLambda_i\delta+e^{-1}\delta+p_i
\delta \|\overline{W}^R_i\|_{[0,T]}+2\|G_i^R\|_{[0,T]}.
\end{equation}
Equation~\eqref{eq:scaled_workload_final} also implies the following upper bound for the workload
process:
\begin{equation} \label{eq:workload_upper_bound}
\overline{W}^R_i(t) \leq \overline{W}^R_i(0)+\upLambda_i t +
G_i^R(t) \leq \|\overline{{W}}^R(0)\|+\|{\upLambda}\|
T + \|G_i^R\|_{[0,T]}.
\end{equation}
By~\eqref{eq:t-s} and~\eqref{eq:workload_upper_bound},
\begin{equation*}
\begin{split}
\omega( \overline{{W}}^R(\cdot), \delta, T) \leq
(\|{\upLambda}\| +e^{-1})\delta \\ +\|{p}\|\delta\left(
\|\overline{{W}}^R(0)\|+\|{\upLambda}\| T+
\|{G}^R\|_{[0,T]}\right)
+2\|{G}^R\|_{[0,T]},
\end{split}
\end{equation*}
which implies that the second condition of Proposition~\ref{tightness} holds, since ${\overline{{W}}^R(0) \Rightarrow
{z}^0}$ and, by Remark~\ref{rem:skorokhod},
$\|{G}^R\|_{[0,T]} \Rightarrow 0$.

\textit{Limiting equations.} Now we show that, if
$\{\overline{W}^R(\cdot); \ R \to \infty \}$
converges weakly in $\mathcal{D}$, then its limit point
$\widetilde{{W}}(\cdot)$ a.s.
\begin{itemize}
\item[(i)] is continuous,
\item[(ii)] does not vanish at $ t > 0 $,
\item[(iii)] satisfies the fluid model equation~\eqref{eq:int_fluid_model}.
\end{itemize}
Then, by uniqueness of FMS's, $\widetilde{{W}}(\cdot)$ a.s. coincides with the FMS starting from ${z}^0$.

(i) Continuity of $\widetilde{{W}}(\cdot)$ follows from Proposition~\ref{tightness}.

(ii) By Lemma~\ref{lem:G2_to_zero}, for any $\varepsilon>0 $ and $n \geq 1$, there exists a~constant ${C(\varepsilon,1/n,n) > 0}$ such that
$
\liminf_{R \to \infty} \vec{P} \{ \inf_{1/n
\leq t \leq n} \| \overline{{W}}^R(t) \|_1 \geq
C(\varepsilon,1/n,n) \} \geq 1-\varepsilon.
$
In the interval~$(0,C(\varepsilon,1/n,n)]$, choose a continuity point $\widetilde{C}(\varepsilon,1/n,n)$ for the
distribution of $\inf_{1/n \leq t \leq n} \|
\widetilde{{W}}(t) \|_1$. The mapping $x(\cdot) \to \inf_{1/n \leq t \leq n} \left\| {x}(t)
\right\|$ is continuous in $\mathcal{D}$; hence,
\begin{align*} 
1 - \varepsilon &\leq \lim\limits_{R \to \infty} \vec{P} \Bigl\{ \inf\limits_{1/n \leq t \leq n} \|
\overline{{W}}^R(t) \|_1 \geq
\widetilde{C}(\varepsilon,1/n,n) \Bigr\}\\
&= \vec{P} \Bigl\{ \inf\limits_{1/n \leq t \leq n} \| \widetilde{{W}}(t) \|_1
\geq \widetilde{C}(\varepsilon,1/n,n) \Bigr\} 
\leq \vec{P} \Bigl\{ \inf\limits_{1/n \leq t \leq n} \| \widetilde{{W}}(t) \|_1
> 0 \Bigr\} .
\end{align*}
By taking the limits as $\varepsilon \to 0$ in the last inequality, we obtain $ \vec{P} \{ \inf_{1/n \leq t \leq n} \| \widetilde{{W}}(t) \|_1
> 0 \}  = 1$ for any $n \geq 1$, which, in turn, implies that
\begin{equation}\label{eq:bounded_away_from_zero}
\vec{P} \Bigl\{ \| \widetilde{{W}}(t) \|_1
> 0 \text{ for all } t > 0 \Bigr\}=1.
\end{equation}

(iii) Fix $t \geq 0$. We introduce the mappings
$\varphi^1_t,\varphi^2_t \colon \mathcal{D} \to
\mathbb{R}^K$ defined by $\varphi^1_t({x})={x}(t) + {p \ast} \int_0^t  {x}(s) ds$,
$\varphi^2_t({x})=e^{-1} \int_0^t 
{m}({x}(s)) ds$. By Remark~\ref{rem:skorokhod}, the mapping
$\varphi^1_t$ is continuous at any continuous ${x}(\cdot)$. By $m_i(\cdot) \leq 1$, $i=1,\ldots,K$, and the dominated convergence theorem, the mapping
$\varphi^2_t$ is
continuous at any continuous ${x}(\cdot)$ that
differs from zero everywhere except the points forming a
set of zero Lebesgue measure. Then, by continuity of
$\widetilde{{W}}(\cdot)$, \eqref{eq:bounded_away_from_zero} and the
continuous mapping theorem, we have
\begin{equation*}
\varphi^1_t(\overline{{W}}^R)+\varphi^2_t(\overline{{W}}^R)
\Rightarrow \varphi^1_t(\widetilde{{W}})+\varphi^2_t(\widetilde{{W}}).
\end{equation*}
On the other hand, by\eqref{eq:scaled_workload_final},
$$
\varphi^1_t(\overline{{W}}^R)+\varphi^2_t(\overline{{W}}^R)=\overline{{W}}^R(0)+t
{\upLambda}+{G}^R(t) \Rightarrow {z}^0+t
{\upLambda}.
$$
Since ${z}^0+t
{\upLambda}$ is deterministic, we have, for any fixed $t \geq 0$,
$$
\vec{P} \Bigl\{  \widetilde{{W}}(t)={z}^0+t
{\upLambda}-e^{-1}\!\!\int_0^t\!\! {m}\left(\!
\widetilde{{W}}(s) \!\right)\!ds-{p
\cdot}\!\!\int_s^t\!\!
\widetilde{{W}}(s)ds \Bigr\}=1.
$$
Let $\upOmega_t$ denote the event in the last equality. Then, again due to continuity of $\widetilde{{W}}(\cdot)$ and~\eqref{eq:bounded_away_from_zero},
\begin{equation*}
\quad \vec{P} \Bigl\{ \widetilde{{W}}(\cdot) \text{ satisfies~\eqref{eq:int_fluid_model} in } [0,\infty) \Bigr \} = \vec{P} \Bigl \{ \bigcap
\upOmega_t \text{ over all rational } t \in [0,\infty) \Bigr\} =1.
\end{equation*}

\subsection{Proof of Lemma~\ref{lem:G2_to_zero}} \label{ssec:proof_lemma3}
We split the proof into four parts. In the first two parts, we show that Assumptions~\ref{ass:nontrivial_arrivals} and~\ref{ass:overload} imply~(i), and that~(i) implies~(ii), both for the single class case. In the third part, we show that the total workload $\| {W}(\cdot) \|_1$ of a~multiclass model is bounded
from below by that of a~single class model with suitable parameters. Then~(i) and~(ii) hold for the multiclass case, too. In the last part, we show that~(ii)
implies~(iii).

\paragraph{Assumptions~\ref{ass:nontrivial_arrivals} and~\ref{ass:overload} imply~(i), single class case}

For every $\gamma$ from an~interval $(0,\gamma^{\, \ast}]$, we construct a~Markov chain (see $\{ V_\gamma(n); \ n \geq 0 \}$ below) that, for all $R$ large enough, is a~lower bound for the workload process until the latter first hits the set $[\gamma R, \infty)$. Then we choose a~$\gamma$ so as to have (i) with $\{ V_\gamma(n); \ n
\geq 0 \}$ in place of $\{ W^R(n); \ n \geq 0 \}$, and this completes the proof.

Throughout the proof, $\delta$ and $\varepsilon$ are fixed.

Without loss of generality, we can assume that, for all $R$, a.s. $W^R(0)=0$. Indeed, 
\begin{equation}\label{eq:monotonicity}
\text{for all } n \text{ and } x \geq y, \quad x-T(n,x)-I^R(n,x)
\geq y-T(n,y)-I^R(n,y) \text{ a.s.}
\end{equation}
Property~\eqref{eq:monotonicity} says that the process
$\{ W^R(n); \ n \geq 0 \}$ admits path-wise monotonicity: the bigger is the initial value $W^R(0)$, the bigger are all the other values $W^R(n)$, $n \geq 1$.

Further we make preparations needed to construct the lower-bound Markov chains. 
Let
\begin{equation} \label{eq:h,N}
h^\ast=e^{-1}+(\upLambda-e^{-1})/2  \quad \text{and} \quad B(k,1/k)(\{1\}) \leq h^\ast, \quad k \geq N.
\end{equation}
Let $\{ B(n); \ n \geq 1 \}$ be an~i.i.d. sequence with $B(1) \sim B(N,p)$.

We apply the following proposition (see the Appendix for the proof) with $a=(\upLambda-e^{-1})/4$.
\begin{proposition} \label{binomial_uniform_bound}
For any $a>0$, there exists a~$\gamma^{\, \ast}=\gamma^{\, \ast}(a)$ and a~family of r.v.'s $\{ \theta_\gamma; \ 0 \leq \gamma \leq \gamma^\ast \}$ with the following
properties:
\begin{itemize}
\item[\textup{(i)}] the family $\{ \theta_\gamma; \ 0 \leq \gamma \leq \gamma^{\, \ast} \}$ is uniformly integrable;
\item[\textup{(ii)}] for any $\gamma \in [0,\gamma^{\, \ast}]$, $\vec{E} \theta_\gamma \leq a$;
\item[\textup{(iii)}] $\theta_\gamma \Rightarrow \theta_{\, 0}$ as $\gamma \to 0$;
\item[\textup{(iv)}] for any $\gamma \in (0,\gamma^{\, \ast}]$, there exists an~$R_\gamma$ such that $\theta_\gamma \geq_{\textup{st}} B_\gamma^R \sim B(\lfloor \gamma R \rfloor,
p/R)$, $R \geq R_\gamma$.
\end{itemize}
\end{proposition}
For $\gamma \in (0,\gamma^{\, \ast}]$, let $\{ \theta_\gamma(n); \ n \geq 1 \}$ be an~i.i.d. sequence with $\theta_\gamma(1) \equal\limits^{\text{d}} \theta_\gamma$,
and assume that this sequence does not depend on ${\{ B(n) ; \ n \geq 1 \}}$.

Now we construct the lower-bound Markov chains. For $R$ large enough, given $W^R(n) < N$, 
\begin{equation}\label{eq:<N}
W^R(n+1) - W^R(n) \greaterequal\limits^{\text{a.s.}} A(n+1) -
1 - \sum_{i=1}^N \xi^R(n,i) \geq_{\text{st}} A(1) - 1 - B(1),
\end{equation}
and, given $N \leq W^R(n) < \gamma R$,
\begin{equation}\label{eq:<aR}
\begin{split}
W^R(n+1) - W^R(n) &\greaterequal\limits^{\text{a.s.}} A(n+1)
- \vec{I} \{ U(n) \leq h^\ast \} - \sum_{i=1}^{\lfloor \gamma R \rfloor} \xi^R(n,i) \\
&\geq_{\text{st}} A(1) - \vec{I} \{ U(0) \leq h^\ast \} - \theta_\gamma.
\end{split}
\end{equation}

Introduce the two i.i.d.
sequences:
\begin{align} \label{eq:y_after_N}
x(n) &= A(n) - 1 - B(n), \quad n
\geq 1, \nonumber \\
y_\gamma(n) &= A(n) - \vec{I} \{ U(n-1) \leq h^\ast \} -
\theta_\gamma(n), \quad n \geq 1, 
\end{align}
and the two auxiliary Markov Chains:
\begin{align*}
V^R_\gamma(0) &= 0,\\
V^R_\gamma(n+1) &= \left\{
\begin{array}{ll}
\max \{ 0, V^R_\gamma(n) + x(n+1) \}& \text{if }V^R_\gamma(n) < N,\\
\max \{ 0, V^R_\gamma(n) + y_\gamma(n+1) \}& \text{if }N \leq V^R_\gamma(n) < \gamma R,\\
\begin{array}{l} V^R_\gamma(n) + A(n+1) \\ - T\left(n,V^R_\gamma(n)\right) -
I^R\left(n,V^R_\gamma(n)\right) \end{array}& \text{if }V^R_\gamma(n) \geq \gamma R,
\end{array}
\right. \\
V_\gamma(0) &=0,\\
V_\gamma(n+1) &= \left\{
\begin{array}{ll}
\max \{ 0, V_\gamma(n) + x(n+1) \}& \text{if }V_\gamma(n) < N,\\
\max \{ 0, V_\gamma(n) + y_\gamma(n+1) \}& \text{if
} V_\gamma(n) \geq N.
\end{array}
\right.
\end{align*}

Put $\psi^R_\gamma(\gamma R)$  and $\psi_\gamma(\gamma R)$ to be the first hitting
times of the set $[\gamma R,\infty)$ for the processes $\{V^R_\gamma(n); \ n
\geq 0\}$ and $\{V_\gamma(n); \ n \geq 0\}$, respectively. Then $\psi^R_\gamma(\gamma R) = \psi_\gamma(\gamma R)$ for all $R$.

The processes $ \{V^R_\gamma(n); \ n \geq 0 \}$ and
$\{ W^R(n); \ n \geq 0 \}$ are related in the following
way: $V^R_\gamma(n)=x$, $W^R(n)=y$, where $x \leq y$, implies
$
W^R(n+1)\geq_{\text{st}}V^R_\gamma(n+1).
$
Indeed, due to inequalities \eqref{eq:monotonicity}, \eqref{eq:<N} and~\eqref{eq:<aR},
\begin{align*}
W^R(n+1) &= y + A(n+1) - T\left(n,y\right) -
I^R\left(n,y\right) \greaterequal\limits^{\text{a.s.}} x + A(n+1)
- T\left(n,x\right) - I^R\left(n,x\right)\\
& \quad \left\{
\begin{array}{ll}
\geq_{\text{st}} \max\{0,x + x(n+1)\}=V^R_\gamma(n+1)& \text{if } x < N,\\
\geq_{\text{st}} \max\{0,x + y(n+1)\}=V^R_\gamma(n+1)& \text{if } N \leq x < \gamma R,\\
= V^R_\gamma(n+1)& \text{if } x \geq \gamma R.
\end{array}
\right.
\end{align*}

Then we get 
$
\varphi^R(\gamma R) \geq_{\text{st}} \psi^R_\gamma(\gamma R) = \psi_\gamma(\gamma R)
$
as a~consequence of the following result (see the Appendix for the
proof).
\begin{proposition}\label{coupling_markov_chains}
Suppose $\{X(n); \ n \geq 0\}$ and $\{Y(n); \ n \geq 0$\} are Markov
chains with a~common state space $\mathcal{S}$, where $\mathcal{S}$ is a~closed subset of $\mathbb{R}$, and deterministic initial states $X(0) \leq Y(0)$.
Suppose also that, for any $x \leq y$ and any $z$,
\[
\vec{P}\{
X(n+1)\geq z \big| X(n)=x \} \leq \vec{P}\{ Y(n+1)\geq z \big|
Y(n)=y \}.
\]
Then there exist Markov Chains $\{\widetilde{X}(n); \ n \geq 0\}$
and $\{\widetilde{Y}(n); \ n \geq 0\}$ defined on a~common
probability space, distributed as $\{X(n); \ n \geq 0\}$ and 
$ \{Y(n); \ n \geq 0\} $, respectively, and such that $\widetilde{X}(n) \leq
\widetilde{Y}(n) \text{ a.s.}$ for all $n$.
\end{proposition}

Now our goal is to choose a~$\gamma \in (0,\gamma^{\, \ast}]$ so as to have
\begin{equation}\label{eq:linear_level_linear_time}
\liminf\limits_{R \to \infty} \vec{P} \left\{ \psi_\gamma(\gamma R) \leq
\delta R \right\} \geq 1-\varepsilon.
\end{equation}

To track the moments when the process $\{ V_\gamma(n); \  n \geq 0 \}$ reaches level
$N$ from below and above, we recursively define the hitting times
\begin{align*}
\tau_\gamma^{(0)} & = 0, & \tau^{(i)}_\gamma & = \inf \{ n \geq \nu^{(i-1)}_\gamma \colon V_\gamma(n) \geq N \}, \quad i \geq 1,\\
\nu_\gamma^{(0)} & = 0, & \nu^{(i)}_\gamma & = \inf \{ n \geq \tau^{(i)}_\gamma \colon V_\gamma(n) < N \}, \quad i \geq 1.
\end{align*}
By convention, infimum over the empty set is $\infty$. So, if either $\tau^{(i)}_\gamma=\infty$, or $\nu^{(i)}_\gamma=\infty$, then
$\tau^{(j)}_\gamma=\nu^{(j)}_\gamma=\infty$ for all $j > i$.

Note that the r.v. $\tau^{(1)} := \tau^{(1)}_\gamma$ is a.s. finite and does not depend on $\gamma$ because, for $n \leq \tau^{(1)}$, the process $\{V_\gamma(n); \ n \geq 0\}$ is a~reflected homogeneous random walk with i.i.d. increments $\{ {x(n);} \  {n \geq 1} \}$, which do not depend on $\gamma$. By Assumption~\ref{ass:nontrivial_arrivals}, $\vec{P} \{
x(1) > 0 \} > 0$, then $\vec{E} \tau^{(1)} < \infty$. Further, for any~$i$, if $\nu^{(i-1)}_\gamma$ is finite, then $\tau^{(i)}_\gamma$ is finite, too, and
the difference $\widetilde{\tau}^{(i)}_\gamma:=\tau^{(i)}_\gamma - \nu^{(i-1)}_\gamma$ is stochastically bounded from above by $\tau^{(1)}$.

Let $q^{(i)}_\gamma = \vec{P} \{ \nu^{(i)}_\gamma < \infty \big| \nu^{(i-1)}_\gamma < \infty \}$. Then there exists
a constant $\widetilde{q}<1$ such that,
\begin{equation}\label{eq:nu_probability_bound}
\text{for all } i \text{ and } \gamma \text{ small enough}, \quad q^{(i)}_\gamma \leq
\widetilde{q}.
\end{equation}
Indeed, for all $\gamma \in [0,\gamma^{\, \ast}]$, consider the random walks $Y_\gamma(n):=\sum_{i=1}^n y_\gamma(i)$ (here $y_0(n)$, $n \geq 1$, are defined
by~\eqref{eq:y_after_N} with $\gamma = 0$). By Proposition~\ref{binomial_uniform_bound}, the family $\{ y_\gamma(1); \ 0 \leq \gamma \leq \gamma^{\, \ast} \}$ is
uniformly integrable, and $y_\gamma(1) \Rightarrow y_0(1)$ as $\gamma \to 0$, which, together with $\vec{E} y_0(1) \geq (\upLambda - e^{-1}) / 4 > 0$, implies that
$\inf_{n \geq 0} Y_\gamma(n) \Rightarrow \inf_{n \geq 0} Y_0(n)$ as $\gamma \to 0$ (see~\cite[Chapter~X,
Theorem~6.1]{newAsmussen}). Also  $\vec{E} y_0(1) > 0$ implies $\vec{P} \{ \inf_{n \geq 0} Y_0(n) \} =: p_0 > 0$ (see~\cite[Chapter~VIII, Theorem~
2.4]{newAsmussen}). Then, for all~$i$, we have $\vec{P} \{ \nu^{(i)}_\gamma = \infty \big| \nu^{(i-1)}_\gamma < \infty \} \leq \vec{P} \{ \inf_{n \geq 0}
Y_\gamma(n) \geq 0 \} \to p_0 > 0$, and, for all~$i$ and~$\gamma$ small enough, $q^{(i)}_\gamma \leq
1 - p_0/2 =: \widetilde{q} < 1$.

Let $K_\gamma = \inf \{ i \geq 1 \colon \nu^{(i)}_\gamma =
\infty \}$. By~\eqref{eq:nu_probability_bound}, the $K_\gamma$'s are stochastically bounded from above by a~geometric r.v. uniformly in~$\gamma$ small enough,
\begin{equation} \label{eq:nu_geometric_distribution}
\vec{P}  \{ K_\gamma > i \} \leq \widetilde {q}^{\:i}, \quad i
\geq 0.
\end{equation}

Further, for $i=1,\ldots,K_\gamma$, define the hitting times
$$
\widetilde{\nu}^{(i)}_\gamma = \inf \{ n \geq 0 \colon
\sum\nolimits_{j=1}^{n} {y_\gamma (\tau^{(i)}_\gamma + j )} \geq \gamma R
\}.
$$
Since $\vec{E} y_\gamma(1) > 0$, these r.v.'s are finite. We have $\psi_\gamma(\gamma R) \leq \sum_{i=1}^{K_\gamma} ( \widetilde{\tau}^{(i)}_\gamma +
\widetilde{\nu}^{(i)}_\gamma )$. Indeed, if $\min \{ {i \geq 1 \colon} \nu^{(i)}_\gamma - \tau^{(i)}_\gamma \geq \widetilde{\nu}^{(i)}_\gamma \} = k$, then $k
\leq K_\gamma$ because $\nu_\gamma^{(K_\gamma)}=\infty$, and
\begin{align*}
\psi_\gamma(\gamma R) &\leq \tau^{(k)}_\gamma + \widetilde{\nu}^{(k)}_\gamma = ( \tau^{(1)}_\gamma-\nu^{(0)}_\gamma ) + ( \nu^{(1)}_\gamma - \tau^{(1)}_\gamma )+ \cdots + ( \tau^{(k)}_\gamma - \nu^{(k)}_\gamma ) + \widetilde{\nu}^{(k)}_\gamma\\
&\leq \widetilde{\tau}^{(1)}_\gamma + \widetilde{\nu}^{(1)}_\gamma + \cdots + \widetilde{\tau}^{(k)}_\gamma + \widetilde{\nu}^{(k)}_\gamma \leq
\widetilde{\tau}^{(1)}_\gamma + \widetilde{\nu}^{(1)}_\gamma + \cdots + \widetilde{\tau}^{(K_\gamma)}_\gamma + \widetilde{\nu}^{(K_\gamma)}_\gamma.
\end{align*}

Now we are ready to complete the proof. By~\eqref{eq:nu_geometric_distribution}, there are $k_0$ and $\widetilde{\gamma}$ such that $\vec{P} \{ K_\gamma >
k_0 \} \leq \varepsilon$ for all $\gamma \leq \widetilde{\gamma}$. Put $\delta_0=\delta/(2k_0)$ and $\gamma=\min\{ \gamma^{\, \ast}, \widetilde{\gamma} , \delta_0
(\upLambda - e^{-1}) / 8 \}$. Then
\begin{equation}\label{eq:psi<t+s}
\begin{split}
\vec{P} \{ \psi_\gamma (\gamma R) > \delta R\}
\leq \vec{P} \Bigl\{ \sum\nolimits_{i=1}^\infty (\widetilde{\tau}^{(i)}_\gamma +
\widetilde{\nu}^{(i)}_\gamma ) \vec{I}\{ K_\gamma \geq i \} >\delta R \Bigr\} &\\
\leq \vec{P} \Bigl\{ \sum\nolimits_{i=1}^{k_0} (\widetilde{\tau}^{(i)}_\gamma + \widetilde{\nu}^{(i)}_\gamma ) \vec{I}\{ K_\gamma \geq i \} >\delta R
\Bigr\} + \varepsilon &\\
\leq \sum\nolimits_{i=1}^{k_0} \underbrace{ \vec{P} \{ \widetilde{\tau}^{(i)}_\gamma \vec{I}\{ K_\gamma \geq i \} > \delta_0 R \}
}_{\displaystyle{=:t_i^R}}+ \sum\nolimits_{i=1}^{k_0} \underbrace{ \vec{P} \{ \widetilde{\nu}^{(i)}_\gamma \vec{I}\{ K_\gamma \geq i \} > \delta_0 R \}
}_{\displaystyle{=:s_i^R}} + \varepsilon.
\end{split}
\end{equation} 

Since $\{ K_\gamma \geq i \} \subseteq \{ \nu^{(i-1)} < \infty \}$ and $\tau^{(1)}$ is a.s. finite,
\begin{equation}\label{eq:t}
\begin{split}
t_i^R & \leq \vec{P} \{ \widetilde{\tau}^{(i)}_\gamma \vec{I}\{ \nu^{(i-1)} < \infty \} > \delta_0 R \} \\
&= \vec{P} \{ \widetilde{\tau}^{(i)}_\gamma >
\delta_0 R | \nu^{(i-1)} < \infty \} \vec{P} \{ \nu^{(i-1)} < \infty \} \\
& \leq \vec{P} \{ \tau^{(1)} > \delta_0 R \} \to 0 \quad \text{as } R \to \infty.
\end{split}
\end{equation} 

For $i=1,\ldots,K_\gamma$, we have $\{ \sum_{j=1}^{\lfloor \delta_0 R \rfloor} y_\gamma( \tau_{\gamma}^{(i)} + j ) \geq \gamma R \} \subseteq \{
\widetilde{\nu}_\gamma^{(i)} \leq \delta_0 R \}$. Then
\begin{equation} \label{eq:s}
\begin{split}
s_i^R & = \vec{P} \{ \widetilde{\nu}^{(i)}_\gamma > \delta_0 R | K_\gamma \geq i \} \vec{P} \{ K_\gamma \geq i \} \\
& \leq \vec{P} \{ \sum\nolimits_{j=1}^{\lfloor \delta_0 R \rfloor} y_\gamma( \tau_{\gamma}^{(i)} + j ) < \gamma R | K_\gamma \geq i \} \\
&= \vec{P} \{
Y_\gamma(\lfloor \delta_0 R \rfloor) < \gamma R \} \to 0 \quad \text{as } R \to \infty
\end{split}
\end{equation} 
because a.s. $Y_\gamma(\lfloor \delta_0 R \rfloor) / R \to \delta_0 \vec{E} y_\gamma(1) \geq \delta_0 (\upLambda - e^{-1}) / 4 > \delta_0 \gamma$.

Finally,~\eqref{eq:psi<t+s}, \eqref{eq:t} and~\eqref{eq:s} imply~\eqref{eq:linear_level_linear_time}.

\paragraph{(i) implies~(ii), single class case }

By~(i), we can choose a~$\gamma > 0$ such that, for large $R$, the process $W^R(\cdot)$ reaches level $\gamma R$ in time $\varphi^R(\gamma R) \leq \delta R$
with high probability. Now we prove that, within the time horizon $[ \varphi^R(\gamma R) , \upDelta R]$, there exists a~minorant for $W^R(\cdot)$ that, for large
$R$, stays close to level $\gamma R$ with high probability. Then $W^R(\cdot)$ stays higher than, for example, level $\gamma R / 2$.

We now proceed more formally. Fix $\delta$, $\upDelta$ and $\varepsilon$. Take $h^\ast$ and $N$ the same as in~\eqref{eq:h,N}. Take $\gamma$ and r.v.~$\theta_\gamma$
that satisfy (i)~of~Lemma~\ref{lem:G2_to_zero}, (ii) (with~$a=(\upLambda - e^{-1}) / 4$) of Proposition~\ref{binomial_uniform_bound} and~(iv) of Proposition~\ref{binomial_uniform_bound}. Let $\{
\theta(n); \ n \geq 1 \}$ be an~i.i.d. sequence with $\theta(1) \equal\limits^{\text{d}} \theta_\gamma$, and assume that this sequence does not depend on $\{ A(n); \ n
\geq 1 \}$ and $\{ U(n); \  n \geq 0 \}$. Let $\{ v^R(n); \  n \geq 1 \}$ and $\{ y(n); \ n \geq 1 \}$ be i.i.d. sequences with $v^R(n) = A(n) - \vec{I} \{ U(n-1)
\leq h^\ast \} - \sum_{i=1}^{\lfloor \gamma R \rfloor}\xi^R(n-1,i)$ and $y(n)=A(n) -
\vec{I} \{ U(n-1) \leq h^\ast \} - \theta(n)$. Define the auxiliary processes
$$
V^R(n)=\left\{
\begin{array}{ll}
W^R(n),& n < \varphi^R(\gamma R),\\
\lfloor \gamma R \rfloor,& n = \varphi^R(\gamma R),\\
\min \left\{ \lfloor \gamma R \rfloor, V^R(n-1){+}v^R(n) \right\},&
n > \varphi^R(\gamma R),
\end{array}
\right.
$$
and 
\begin{align*}
\widetilde{V}^R(0) & = 0, & \widetilde{V}^R(n) & =  \min \{ 0, \widetilde{V}^R(n-1)+v^R(n) \}, \quad n \geq 1,\\
Y(0) & = 0, & Y(n) & = \min \{ 0, Y(n-1)+y(n) \}, \quad n \geq 1.
\end{align*}

The processes $W^R(\cdot)$ and $V^R(\cdot)$ coincide within the time interval $[0, \varphi^R(\gamma R) - 1]$. Starting from time $\varphi^R(\gamma R)$, as long
as $V^R(\cdot)$ stays above level $N$, it stays a~minorant for $W^R(\cdot)$. Indeed, for $R$ large enough, given $N \leq V^R(i) \leq
W^R(i)$, $i=\varphi^R( \gamma R), \ldots, n$, if $W^R(n) \geq \lfloor \gamma R \rfloor$, then, by~\eqref{eq:monotonicity},
\begin{align*}
W^R(n+1) &\greaterequal\limits^{\text{a.s.}} \lfloor \gamma R \rfloor + A(n+1)- T\left(n, \lfloor \gamma R
\rfloor \right) - I^R \left(n, \lfloor  \gamma R \rfloor \right)\\
& = \lfloor \gamma R \rfloor + v^R(n+1) \geq
V^R(n+1),
\end{align*}
and, if $W^R(n) < \lfloor \gamma R \rfloor$, then
\begin{align*}
W^R(n+1) &\geq W^R(n) + A(n+1) - \vec{I} \{U(n) \leq
h^\ast \} - \sum\nolimits_{i=1}^{\lfloor \gamma R
\rfloor}\xi^R(n,i) \\
& \geq V^R(n)+ v^R(n+1) \geq V^R(n+1).
\end{align*}

Further, by independence arguments, for $R$ large enough, $y(1) \leq_{\text{st}} v^R(1)$, and
$
\min_{0 \leq i \leq n} Y(i) \leq_{\text{st}} \min_{0 \leq i \leq n} \widetilde{V}^R(i)
$.
Since $\vec{E}y(1)>0$, $\min_{0 \leq i \leq n} Y(i) / n \to 0$ a.s. Hence
\begin{equation}\label{eq:V^R_to_zero}
\min_{0 \leq i \leq
\lfloor \upDelta R \rfloor} \widetilde{V}^R(i) / R \Rightarrow 0 \quad \text{as } R \to \infty.
\end{equation}

Now we are ready to complete the proof. Put $C=\gamma/2$ and define the events
\begin{align*}
E^R &= \{ \min\nolimits_{\lfloor \delta R \rfloor \leq
n \leq \lfloor \upDelta R \rfloor} W^R(n) < R C \}, \displaybreak[0] \\ 
A^R &= \{ \varphi^R(\gamma R) \leq  \delta R \}, \displaybreak[0] \\ 
B^R &= \{ \min\nolimits_{0 \leq i \leq \lfloor \upDelta R \rfloor} V^R ( \varphi^R(\gamma R) + i ) \geq 3 \gamma R / 4 \}.
\end{align*}

Then $\vec{P} \{ E^R \} \leq \vec{P} \{ E^R \cap A^R \cap B^R \} + \vec{P} \{ \overline{A^R} \} + \vec{P} \{ \overline{B^R} \},$ where
\begin{itemize}
\item[$\bullet$] $ E^R \cap A^R \cap B^R \subseteq \{ 3\gamma R / 4 \leq \min_{\lfloor
\delta R \rfloor \leq n \leq \lfloor \upDelta R \rfloor} W^R(n) < \gamma R /2 \} = \varnothing$,
\item[$\bullet$] $\limsup_{R \to \infty} \vec{P} \{
\overline{A^R} \} \leq \varepsilon$,
\item[$\bullet$] by $\{ V^R ( \varphi^R(\gamma R) +n); \ n \geq 0
\}\equal^\text{d} \{ \widetilde{V}^R(n)+\lfloor \gamma R
\rfloor; \ n\geq 0 \}$ and~\eqref{eq:V^R_to_zero}, $$\vec{P} \{
\overline{B^R} \} = \vec{P} \{
\min\nolimits_{0 \leq n\leq \lfloor \upDelta R \rfloor} \widetilde{V}(n) <
3 \gamma R / 4 - \lfloor \gamma R \rfloor \} \to 0 \quad \text{as } R \to \infty.$$
\end{itemize}
Hence, (ii) of Lemma~\eqref{lem:G2_to_zero} holds.

\paragraph{Single class bound for a~multiclass model}

Now we show that a~model with multiple classes of customers can be coupled with a~suitable single class model in such a~way that the workload process of the
single class model is majorised by the total workload of the multiclass model within the whole time horizon $[0,\infty)$. This, in particular, implies that
statements~(i) and~(ii) of the lemma, proven in the single class case, are valid in the multiclass case, too.

For the multiclass model, we slightly modify the
representation of the workload process suggested in Section~\ref{ssec:representation}. We only change the terms that represent departures due to impatience.
For ${x} \in \mathbb{Z}_+^K$, let
$$
I^R_i(n,{x})=\!\!\!\!\!
\sum\limits_{j=x_1 - T_1(n,{x})+ \ldots
+x_{i-1} - T_{i-1}(n,{x})}^{x_1-T_1(n,{x})+ \ldots +
x_i-T_i(n,{x})} \!\!\!\!\! \vec{I} \left\{ U(n,j) \leq
\dfrac{p_i}{R} \right\},
$$
where the r.v.'s $U(n,i)$, $i \geq 1$, $n \geq 0$, are mutually
independent and distributed
uniformly over the interval $[0,1]$. We also assume that the $U(n,i)$'s do not depend on the
random elements ${W}^R_0$, $\left\{{A}(n); \ n \geq 1
\right\}$ and $\left\{U(n); \ n \geq 0\right\}$.

Consider a~single class model with
\begin{itemize}
    \item[$\bullet$] initial condition $\widetilde{W}^R(0)=\left\| {W}^R(0)
    \right\|_1$,
    \item[$\bullet$] arrival process $\widetilde{A}(n)=\left\| {A}(n)
    \right\|_1$,
    \item[$\bullet$] reneging probability
    $\widetilde{p}=\max_{1 \leq i \leq K}p_i$,
\end{itemize}
and define its dynamics as follows:
\begin{equation}
\widetilde{W}^R(n+1) = \widetilde{W}^R(n) + \widetilde{A}(n+1) - \widetilde{T}\left(n,\widetilde{W}^R(n)\right) - \widetilde{I}^{\,R}\left(n, \widetilde{W}^R(n) \right), \label{eq:aux_single_class}
\end{equation}
where, for $k \in \mathbb{Z}_+$,
\begin{align*} \widetilde{T}(n,k) &= \vec{I} \left\{ U(n) \leq h(k) \right\}, \\
\widetilde{I}^{\,R}(n,k) &= \sum\nolimits_{i=1}^{k-\widetilde{T}(n,k)} \vec{I} \left\{ U(n,i) \leq
\dfrac{\widetilde{p}}{R} \right\},
\end{align*}
and the r.v.'s $U(n)$, $U(n,i)$, $n \geq 0$, $i \geq 1$, are those defining the multiclass model.
Then, in particular,
\begin{equation} \label{eq:transmissions=transmissions}
\begin{split}
\left\| {T}(n,{x}) \right\|_1 \equal\limits^{\text{a.s.}} \vec{I} \Bigl\{
U(n) \leq \sum\nolimits_{j=1}^K p_j({x})
\Bigr\}=\widetilde{T}(n,\|{x}\|_1), & \\
\vec{I}\Bigl\{ U(n,j)\leq \dfrac{p_i}{R} \Bigr\} \leq
\vec{I}\Bigl\{ U(n,j)\leq \dfrac{\widetilde{p}}{R} \Bigr\}.&
\end{split}
\end{equation}

We show by induction that $\widetilde{W}^R(\cdot)$ bounds $\| {W}^R(\cdot) \|_1$ from below. Let  $N^R(n) = \| {W}^R(n) \|_1 - T\left(n, \| {W}^R(n)
\|_1\right) $, and suppose that $\widetilde{W}^R(n) \leq \|
{W}^R(n) \|_1$ a.s., then
\begin{align*}
\| {W}^R(n+1) \|_1
&\greaterequal\limits^{\text{a.s.}} \| {W}^R(n)
\|_1 + \widetilde{A}(n+1) - \widetilde{T}\left(n, \| {W}^R(n)
\|_1\right)-\sum_{i=1}^{N^R(n)} \vec{I}\{
U(n,i) \leq \widetilde{p}/R \}\\
&= \| {W}^R(n) \|_1+\widetilde{A}(n+1)-\widetilde{T}\left(n, \|
{W}^R(n) \|_1\right)-\widetilde{I}^{\, R}\left(n, \|
{W}^R(n) \|_1\right)\\
&\greaterequal\limits^{\text{a.s.}} \widetilde{W}^R(n)+\widetilde{A}(n+1)-\widetilde{T}\left(n,
\widetilde{W}^R(n)\right)-\widetilde{I}^{\, R}\left(n, \widetilde{W}^R(n)\right)=\widetilde{W}^R(n+1),
\end{align*}
where the first and last inequalities hold by~\eqref{eq:transmissions=transmissions} and~\eqref{eq:monotonicity} respectively, and the identity is due to
representation~\eqref{eq:aux_single_class}.

\paragraph{(ii) implies~(iii)}

By Remark~\ref{rem:skorokhod}, it is enough to show that, for any $\upDelta>0$, $i=1,\ldots,K$
\begin{equation}\label{eq:G2_to_zero}
\| G_i^{2,R} \|_{[0,\upDelta]} \Rightarrow 0 \quad \text{as } R \to \infty.
\end{equation}
Fix $\upDelta$, $i$. Recall that
\begin{equation} \label{eq:expression_G2}
G_i^{2,R}(t)=e^{-1} \!\! \int\limits_0^t \!\! m_i^R(s) ds - \biggl(1-\dfrac{p_i}{R}\biggr) \!\!\!\!\! \int\limits_0^{\lfloor Rt \rfloor / R} \!\!\!\!\! h^R(s)
m_i^R(s) ds,
\end{equation}
where $m_i^R(s)=m_i \bigl( R\overline{{W}}^R(s) \bigr)$ and $h^R(s)=h\bigl( R \|\overline{{W}}^R(s) \|_1 \bigr)$. First, we estimate the
subtractor in~\eqref{eq:expression_G2}. Since $R \|\overline{{W}}^R(\cdot) \|_1$ is integer-valued and
non-negative, $h^R(\cdot) \leq 1$. Also $m_i^R(\cdot) \leq 1$. Then, for $t \leq \upDelta$, we have
\begin{equation*}
\begin{split}
& \quad \biggl| \int\limits_0^t \!\! h^R(s) m_i^R(s) \, ds - \biggl(1-\dfrac{p_i}{R}\biggr) \!\!\!\!\! \int\limits_0^{\lfloor Rt \rfloor / R} \!\!\!\!\! h^R(s)
m_i^R(s) \, ds \, \biggr| \\
=& \ \int\limits_{\lfloor Rt \rfloor / R}^t \!\!\!\!\! h^R(s) m_i^R(s) \, ds + \dfrac{p_i}{R}\!\!\! \int\limits_0^{\lfloor Rt \rfloor / R} \!\!\!\!\!
h^R(s) m_i^R(s) \, ds \leq \dfrac{1}{R} +
\dfrac{p \upDelta}{R}.
\end{split}
\end{equation*}
Take $\delta < \upDelta$, then
\begin{equation} \label{eq:bound_G2}
\begin{split}
\| G_i^{2,R} \|_{[0,\upDelta]} & \leq  \dfrac{1+p_i
\upDelta}{R} + \sup_{0 \leq t \leq \upDelta} \biggl|  \int\limits_0^t m_i^R(s) (e^{-1}-h^R(s)) \, ds  \, \biggr| \\
& \leq  \dfrac{1+p_i
\upDelta}{R} + (e^{-1} + 1) \delta + \upDelta \sup\limits_{\delta \leq s \leq \upDelta} |e^{-1} - h^R(s)|. 
\end{split}
\end{equation}

Now we show that
\begin{equation} \label{eq:bound_G2_to_zero}
x^R(\delta,\upDelta):= \| e^{-1} - h^R(\cdot) \|_{[\delta,\upDelta]} \Rightarrow 0 \quad \text{as } R \to \infty.
\end{equation}
For any $\sigma> 0$ and $\varepsilon > 0$ and $C(\delta,
\upDelta,\varepsilon)$ satisfying~\eqref{eq:suf_cond_G2_to_zero},
\begin{align*}
\bigl\{ x^R(\delta,\upDelta) \geq \sigma \bigr\} \subseteq & \Bigl\{ x^R(\delta,\upDelta) \geq \sigma,
\inf_{\delta \leq s \leq \upDelta} \|\overline{{W}}^R(s) \|_1 \geq
C(\delta, \upDelta,\varepsilon) \Bigr\}    \\
& \cup    \Bigl\{ \inf_{\delta
\leq s \leq \upDelta} \|\overline{{W}}^R(s) \|_1 < C(\delta,
\upDelta,\varepsilon) \Bigr\} \\
\subseteq& \Bigl\{ \sup_{s \geq R \, C(\delta,
\upDelta,\varepsilon)} | e^{-1} - h(s) | \geq \sigma
\Bigr\} \\
&\cup \Bigl\{ \inf_{\delta \leq s \leq \upDelta}
\|\overline{{W}}^R(s) \|_1 < C(\delta, \upDelta, \varepsilon) \Bigr\}.
\end{align*}
Here the first event in the very RHS is empty for $R$ large enough, and hence
$$
\limsup\limits_{R \to
\infty} \vec{P} \bigl\{ x^R(\delta,\upDelta) \geq \sigma \bigr\}
\leq \limsup\limits_{R \to \infty} \vec{P} \Bigl\{ \inf_{\delta
\leq s \leq \upDelta} \|\overline{{W}}^R(s) \|_1 < C(\delta, \upDelta,
\varepsilon) \Bigr\} \leq \varepsilon.
$$
Since in the last inequality $\varepsilon > 0$ is arbitrary, for any $\sigma > 0$, as $R \to \infty$, we have $\vec{P} \{ x^R(\delta,\upDelta) \geq \sigma \} \to 0$ , which is~\eqref{eq:bound_G2_to_zero}.

Finally,~\eqref{eq:bound_G2} and~\eqref{eq:bound_G2_to_zero} imply~\eqref{eq:G2_to_zero}.

\subsection{Proof of Lemma~\ref{lem:G3_martingale_to_zero}} \label{ssec:proof_lemma4}

We prove the result in the single class case. The same proof is valid for each coordinate in the~multiclass case.

\paragraph{Convergence of $G^{3,R}(\cdot)$}

By Remark~\ref{rem:skorokhod}, it suffices to show that, for any $T > 0$, 
\begin{equation}\label{eq:G3_to_zero}
\mu^R(T):=\sup_{0 \leq t \leq T}
\int_{\lfloor Rt \rfloor / R}^t \!\!\!\! \overline{W}^R(s) \, ds
\Rightarrow 0  \quad \text{as }R \to \infty.
\end{equation}
Since $\overline{W}^R(\cdot)$ is a~constant function within the
interval $\bigl[\lfloor Rt \rfloor/R , t \bigr]$, we have 
$$
\mu^R(T) = \sup_{t \leq T} \overline{W}^R(t) \bigl( Rt
 - \lfloor Rt \rfloor \bigr) / R \leq \sup_{0 \leq t \leq T}
 \overline{W}^R(t) / R = W^R(0)/R^2+ \sum_{i=1}^{\lfloor TR \rfloor} A(i) / R^2,$$
which implies~\eqref{eq:G3_to_zero}.

\paragraph{Convergence of $\overline{M}^R(\cdot)$}

We represent the martingale $\{ M^R(n); n \geq 0 \}$ as a~sum of three other zero-mean martingales,
\begin{align*}
M^R(n) & = S^R(n) - I^R(n) - T^R(n), \displaybreak[0] \\
S^R(n) & = \sum_{i=1}^n A(i) - n\upLambda, \displaybreak[0] \\
I^R(n) &= \sum\limits_{i=0}^{n-1}\sum\limits_{m=1}^{W^R(i)} \Bigl( \xi^R(i,m) - \dfrac{p}{R} \Bigr),  \displaybreak[0] \\
T^R(n) & = \sum\limits_{i=0}^{n-1}\Bigl( T\left(i,W^R(i)\right) - h\left( W^R(i) \right) \Bigr) + \sum\limits_{i=0}^{n-1}
\!\!\! \sum\limits_{\substack{m=W^R(i) \\ - T\left(i,W^R(i)\right)+1}}^{W^R(i)} \!\!\!\!\!\!\!\! \Bigl(  \xi^R(i,m)-\dfrac{p}{R}h\!\left(W^R(i)\right)
\Bigr).
\end{align*}
It suffices to show that, for $N \in \mathbb{Z}_+$, $N \to \infty$,
\begin{equation} \label{eq:three_in_one}
\max_{1 \leq n \leq N} \dfrac{|X^N(n)|}{N} \Rightarrow 0, \quad X=S,T,I.
\end{equation}

For $X=S$,~\eqref{eq:three_in_one} follows from the functional law of large numbers.

For all $n$ and $N$, we have $| T^N(n+1) - T^N(n) | \leq 4$. Then we get~\eqref{eq:three_in_one} with $X=T$ by applying Doob's inequality for
non-negative submartingales and the following known result (see e.g.~\cite[Chapter~VII, Paragraph~9, Theorem~3]{Feller2}).
\begin{proposition}
Let $\{ X(n); n \geq 1 \}$ be a~sequence of r.v.'s such that, for all $n$, $$\vec{E}\{X(n)|X(1),\ldots,X(n-1)\}=0.$$ Suppose that $b(1) < b(2) < ... \to \infty$ and that
$\sum_{k=1}^\infty b(k)^{-2} \vec{E} X(k)^2 < \infty$. Then a.s. ${b(n)^{-1} \sum_{k=1}^n X(k) \to 0}$.
\end{proposition}

Now we prove~\eqref{eq:three_in_one} for $X=I$. The key tool of this proof is Markov's inequality. We have to show that, for any $\varepsilon > 0$, as $N \to \infty$,
\begin{subequations} \label{eq:min_max}
\begin{align}
\vec{P} \{ \max_{1 \leq n \leq N} I^N(n) / N > \varepsilon\} & \to 0, \label{eq:max}\\
\vec{P} \{ \min_{1 \leq n \leq N} I^N(n) / N < -\varepsilon\} & \to 0. \label{eq:min}
\end{align}
\end{subequations}
By Taylor's expansion, there exists an~$\alpha^\ast > 0$ such that, for any $\alpha \in [0,\alpha^\ast]$, $\rho \in[0,1]$ and r.v.~$\xi(\rho)$ with
$\vec{P}\{ \xi(\rho)=1 \} = \rho = 1 - \vec{P}\{ \xi(\rho)=0 \}$,
\begin{equation} \label{eq:bernoulli_exp_bound}
\vec{E}e^{\alpha(\xi(\rho) - \rho)} \leq e^{\alpha^2 \rho}, \quad \text{and} \quad \vec{E}e^{\alpha(\rho - \xi(\rho) )} \leq e^{\alpha^2 \rho}.
\end{equation}
Since $W^N(n) \leq W^N(0)+\sum_{i=1}^N A(i)$, $n \leq N$, and $W^N(0) / N \Rightarrow z^0$, and $\sum_{i=1}^N A(i) / N \to \upLambda$ a.s., for any $\delta > 0$, there exists an
$M(\delta) \in \mathbb{Z}_+$ such that
\begin{equation} \label{eq:workload_bound_arrivals}
\limsup_{N \to \infty} \vec{P} \{ \max_{0 \leq n \leq N}W^N(n) > M(\delta) N \} \leq \delta.
\end{equation}
Denote the event in~\eqref{eq:workload_bound_arrivals} by $E^N(\delta)$ and put
\begin{equation} \label{eq:alpha_bernoulli}
\alpha(\delta) = \min\{ \alpha^\ast, \varepsilon /(2 M(\delta)) \}.
\end{equation}
We introduce the auxiliary martingale
$$
I^{N,\delta}(n) = \sum\limits_{i=0}^{n-1} \sum\limits_{m=1}^{W^{N,\delta}(i)} \Bigl( \xi^N(i,m) - \dfrac{p}{N} \Bigr), \quad n \geq 1,
$$
where $W^{N,\delta}(i)=\max\{W^N(i),M(\delta) N\}$. Note that, on $\overline{E^N(\delta)}$, we have $I^{N,\delta}(n) = I^N(n)$, $n=0, \ldots, N$. Hence,
\begin{equation} \label{eq:many_t_delta}
\begin{split}
\vec{P} \bigl\{ \max_{1 \leq n \leq N} I^N(n) / N > \varepsilon \bigr\} & \leq \vec{P} \bigl\{ \max_{1 \leq n \leq N} I^{N,\delta}(n) > \varepsilon N
\bigr\} + \vec{P} \bigl\{ E^N(\delta) \bigr\} \\
& \leq \sum\nolimits_{n=1}^N \vec{P} \bigl\{  I^{N,\delta}(n) > \varepsilon N \bigr\} + \vec{P} \bigl\{ E^N(\delta) \bigr\}.
\end{split}
\end{equation}
By Markov's inequality, \eqref{eq:alpha_bernoulli} and \eqref{eq:bernoulli_exp_bound}, for $n \leq N$,
\begin{equation} \label{eq:t_delta}
\begin{split}
\exp(\alpha(\delta) \varepsilon N) \vec{P} \bigl\{  I^{N,\delta}(n) > \varepsilon N \bigr\} \leq \vec{E} \exp( \alpha(\delta) I^{N,\delta}(n) )& \\
= \vec{E} \biggl[ \vec{E} \biggl[ \prod_{i=0}^{n-1} \prod_{m=1}^{W^{N,\delta}(i)} \exp \Bigl( \alpha(\delta) ( \xi^N(i,m) - p / N ) \Bigr) \bigg| W^N(0),\ldots, W^N(n-1)
\biggr] \biggr] &\\
\leq \vec{E} \biggl[ \exp \Bigl(  \alpha^2(\delta) \dfrac{p}{N} M(\delta) N n \Bigr) \biggr] \leq \exp \bigl( \alpha^2(\delta) M(\delta) N \bigr) &.
\end{split}
\end{equation}
By~\eqref{eq:workload_bound_arrivals} and~\eqref{eq:alpha_bernoulli}, bounds~\eqref{eq:many_t_delta} and~\eqref{eq:t_delta} imply that
$$
\limsup_{N \to \infty} \vec{P} \{ \max_{1 \leq n \leq N} I^N(n) / N > \varepsilon\} \leq \delta,
$$
where $\delta>0$ is arbitrary. Hence, convergence~\eqref{eq:max} holds.
Convergence~\eqref{eq:min} can be treated similarly.

\section*{Appendix}
\appendix
\textit{Proof of Proposition~\ref{binomial_uniform_bound}.} There exists an~$R^\ast>0$ such that, for all $n \in \mathbb{Z}_+$, $\gamma \in (0,1]$ and $R \geq
R^\ast$,
$$
e^n\vec{P} \{B^R_\gamma \geq n \} \leq \vec{E} \exp(B^R_\gamma )=(ep/R+1-p/R)^{\lfloor \gamma R\rfloor} \leq \exp((e-1)p)=:\mu.
$$
Take $N^\ast$ such that $\sum_{n > N^\ast} \mu n e^{-n} \leq a/2$, and $\gamma^\ast = \min \{ 1, a/(4p) \}$. Fix $\gamma \in (0,\gamma^\ast]$. Since, as $R \to \infty$, $B(\lfloor
\gamma R \rfloor, p/R) \Rightarrow \upPi(\gamma p)$, there exists an~$R_\gamma \geq R^\ast$ such that $$\vec{P} \{ B^R_\gamma = n \} \leq 2
\upPi(\gamma R)(\{n\}), \quad \text{for } R \geq R_\gamma \text{ and } n=0, \ldots N^\ast. $$
For $\gamma \in (0,\gamma^\ast]$, put
\begin{align} \label{eq:after_N}
\vec{P} \{ \theta_\gamma \geq n \} &= \mu e^{-n}, \quad n > N^\ast, \\
\vec{P} \{ \theta_\gamma = n \} &= \min \bigl\{ 2 \upPi(\gamma R)(\{n\}), 1 - \vec{P} \{ \theta_\gamma \geq n +1 \}\bigr\}, \quad n=N^\ast,\ldots,0.
\nonumber
\end{align}
For $\gamma=0$, put~\eqref{eq:after_N} and 
\begin{align*}
\vec{P} \{ \theta_0 \geq 0 \} &= 1 - \mu e^{-N^\ast-1},\\
\vec{P} \{ \theta_0 = n \} &= 0, \quad n=1,\ldots,N. \tag*{\qed}
\end{align*}

\textit{Proof of Proposition~\ref{coupling_markov_chains}.} Define
\begin{align*}
P(x,\ {\geq} z) &= \vec{P} \left\{ X(n+1) \geq z \big| X(n)=x\right\},\\
Q(y,\ {\geq} z) &= \vec{P} \left\{ Y(n+1) \geq z \big|
Y(n)=y\right\}.
\end{align*}
Let $\left\{U(n); \ n \geq 0\right\}$ be an~i.i.d. sequence with
$U(0)$ distributed uniformly over $[0,1]$. Then put
\begin{align*}
\widetilde{X}(0)  & =  X(0), & \widetilde{X}(n+1) &= \sup \{ z \in \mathcal{S}
\colon U(n) \geq 1-P(\widetilde{X}(n),\ {\geq} z) \},\\
 \widetilde{Y}(0) & = Y(0), & \widetilde{Y}(n+1) &= \sup \{ z \in \mathcal{S}
\colon U(n) \geq 1-P(\widetilde{Y}(n),\ {\geq} z) \}. \tag*{\qed}
\end{align*}

\end{document}